\documentclass{amsart}
\usepackage{psfrag,graphicx,epsf}
\usepackage{amsmath,amsthm,epsf,amscd,amssymb,verbatim}
\usepackage{palatino}
\usepackage[mathbf,mathcal]{euler}
\usepackage{epsfig}

\newtheorem{theorem}{Theorem}
\newtheorem{lemma}[theorem]{Lemma}
\newtheorem{cor}[theorem]{Corollary}
\newtheorem{prop}[theorem]{Proposition}

\theoremstyle{definition}

\theoremstyle{remark}

\numberwithin{equation}{section} \font\bbb=msbm10 scaled 1100

\newcommand{\ncmnd}{\newcommand}
\ncmnd{\nthm}{\newtheorem}

\nthm{thm}[theorem]{Theorem}
\nthm{lem}[theorem]{Lemma}
\nthm{propo}[theorem]{Proposition}

\theoremstyle{definition}
\nthm{defn}[theorem]{Definition}
\nthm{examp}[theorem]{Example}

\theoremstyle{remark}
\nthm{rem}[theorem]{Remark}

\newcommand{\real}{\mbox{\bbb R}}

\newcommand{\catzero}{{\mbox{\sc cat($0$)}}}
\newcommand{\catk}{{\mbox{\sc cat($K$)}}}
\newcommand{\catone}{{\mbox{\sc cat($1$)}}}

\begin{document}

\title[TIME-DEPENDENT GRADIENT CURVES]{TIME-DEPENDENT GRADIENT CURVES ON \catzero\ spaces}
\author{C. Jun}
\address{Department of Mathematics,
        University of Pennsylvania, Philadelphia PA, 19104}
\email{cjun@math.upenn.edu}

\begin{abstract}
We prove existence and uniqueness of time-dependent gradient curves for time-dependent functions with a convexity property on \catzero\ spaces.
As an application, we prove existence and uniqueness of continuous pursuit curves, where the evader can be represented by
a convex set or we can chase the barycenter of multiple evaders.
\end{abstract}



\maketitle

\section{Introduction}
This paper studies time-dependent gradient curves of time-dependent, almost convex (precisely, $\lambda$-convex) functions.
The time-dependent case poses new challenges that do not arise in the time-independent setting.

Time-independent gradient flows have been studied extensively on \catzero\ spaces by Mayer \cite{m} and on related metric spaces by Ambrosio, Gigli and Savar$\mathrm{\acute{e}}$ \cite{ags} \cite{c}. Independently, it was studied in \catk\ spaces geometrically by Lytchak \cite{l}. In a space with curvature bounded from below, it was studied by Petrunin \cite{p}.
In this paper we have invented a method, using the properties proved by Mayer about the gradient flow for a time-independent function, to obtain suitable approximations of the time-dependent case by piecewise fixed-time gradient segments, and show convergence without repeating the very complicated Crandall-Liggett scheme.

After we introduce Mayer's results and modify them slightly in Section 4, we will use time-independent gradient curves to generate discrete solutions for the time-dependent case in Section 5. In order to prove existence, uniqueness and convergence estimates for time-dependent gradient curves, we must formulate appropriate dependence in the time variable (see Examples \ref{ex1111} and \ref{ex111}).

A \emph{\catk\ space} is a complete metric space such that no triangle is fatter than the triangle with same edge lengths in the model space of constant curvature $K$ (see Definition \ref{de3}). Like Mayer, we work on \catzero\ spaces, called nonpositively curved spaces or NPC spaces in \cite{m}.
For $K=0$, we can define \catzero\ spaces by convexity of squared distance functions (see Proposition \ref{g1}).
\catk\ spaces include many important examples.
Among many references, we mention \cite{bh} and \cite{bbi}.
Examples of \catzero\ spaces include simply connected Riemannian manifolds with non-positive sectional curvature (possibly with boundary satisfying a certain condition \cite{abb}) and trees. Spheres, surfaces of revolution, closed Euclidean domains with smooth boundary supported by spheres \cite{abg2} and finite-dimensional spherical polyhedra with the link condition of Gromov \cite{g}, as well as all \catzero\ spaces, are examples of \catk\ spaces for $K>0$.

Our results on time-dependent gradient flows in turn feed back to pursuit-evasion games.
For example, given a point $x_0$, the gradient flow of $dist_{x_0}$ defined by $dist_{x_0}(x):=d(x,x_0)$ is the geodesic flow toward $x_0$ as center. If the point moves, we will get time-dependent gradient curves. For a curve $E=E(t)$, we will show that there are time-dependent gradient curves for the time-dependent function $dist_{E(t)}$ in Section 6. Those curves are called \emph{(simple) continuous pursuit curves} and $E$ is called the \emph{evader}. As we shall show,
we may also allow multiple evaders or uncertainty in evader position.

With different domains and different strategies, pursuit-evasion games have been considered by many mathematicians, computer scientists and engineers. The problems are generated from robotics, control theory and computer simulations. Under a simple pursuit strategy, the main constraints on pursuit-evasion are the geometry and topology of playing domains. Almost always these have been two-dimensional Euclidean domains, or higher-dimensional convex Euclidean domains. Recently there have been results on surfaces of revolution \cite{hm}, cones \cite{m2}, and round spheres \cite{k}. Finally, \catk\ spaces were studied as a natural setting by Alexander, Bishop and Ghrist, because pursuit-evasion requires neither smoothness nor being locally Euclidean \cite{abg2} \cite{abg}.
\catk\ spaces include all previously studied domains and are vastly more general than have been usual in the extensive pursuit-evasion literature.


Recently, continuous pursuit games were applied to show the non-existence of shy-coupled Brownian motions in many Euclidean domains \cite{bbk}.


Very recently (after completion our work on this paper) we discovered Kim and Masmoudi's paper \cite{km} which is close to our work.
However, our paper covers cases not covered by \cite{km}.  In particular, in our
paper, we show that we get time-dependent gradient curves, which are
defined on $[t_0, \infty  )$ without an assumption that our flow has speed bounded uniformly
on the entire space or that our flow has linear speed growth, whereas in \cite{km} they
need to use a Lipschitz constant with uniformly bounded speed.   In Example \ref{ex11111111} we give
a specific example that is covered by our result but not by \cite{km}.


\subsection{Outline of paper}
 In Sections 2 and 3, we list several properties of \catk\ spaces, $\lambda$-convex functions and gradient vectors. In Section 4,
 we look at Mayer's work showing existence and uniqueness of time-independent gradient curves.
 Suppose that we have two $\lambda$-convex functions $F_{t_1}$ and $F_{t_2}$ for fixed-time $t_1$ and $t_2$. Then in Theorem \ref{g21}, we examine the distance between two fixed-time gradient curves issuing from the same point. In this theorem, we need the H\"{o}lder continuity of gradient vectors in the time variable. We give an example illustrating the failure of the extendibility of time-dependent gradient curves in the absence of such a condition.
Section 5 contains the statement and the proof of our main Theorem \ref{g7} showing existence and uniqueness of time-dependent gradient curves. In Section 6, Theorem \ref{g7} is applied to pursuit problems.
In Section 7, we see a property of time-independent gradient curves.

\section{\catk\ spaces}
\subsection{Metric spaces}
Let $(X,d)$ be a metric space. A curve $\gamma : I \to X$ is called a \emph{geodesic}  if for all $t, t^\prime \in I$, $d(\gamma(t),\gamma(t^\prime))=c|t-t^\prime|$ where $c$ is a constant, the \emph{speed} of the geodesic $\gamma$.

$[x y]$ denotes a unit-speed geodesic $\gamma$ from $x$ to $y$ defined on $[0,t]$, where $\gamma(0)=x$, $\gamma(t)=y$ and $t=d(x,y)$.
$\triangle xyz$ denotes the geodesic triangle of geodesics $[x y]$,$[x z]$ and $[y z]$.

A metric space is a \emph{geodesic space} if any two points are joined by a geodesic; and a \emph{C-geodesic space} if any two points with distance $< C$ are joined by a geodesic.

\subsection{\catk\ spaces}
$M_K$ denotes the 2-dimensional, complete, simply-connected space of constant curvature $K$.
Then $M_0 = \mathbb{E}^2$, $M_1 = \mathbb{S}^2$ and $M_{-1} = \mathbb{H}^2$.
Let $d_K$ be the metric of $M_K$.
$D_K$ denotes the diameter of $M_K$. Thus, $D_K = {\pi \over \sqrt{K} }$ if $K > 0$ and $D_K = \infty$ if $K \leq 0$.

A triangle $\triangle \widetilde{x}_1\widetilde{x}_2\widetilde{x}_3$ in $M_K$ is called a \emph{comparison triangle} for $\triangle x_1x_2x_3$ in $X$ if $d_K(\widetilde{x}_i,\widetilde{x}_j)=d(x_i,x_j)$ for $i,j \in \{1,2,3\}$. We write $$\widetilde{\triangle} x_1 x_2 x_3 = \triangle \widetilde{x}_1\widetilde{x}_2\widetilde{x}_3.$$

\begin{defn}\label{de3}
Let $(X,d)$ be a metric space (not necessarily locally compact) and $K$ be a real constant. A complete $D_K$-geodesic space $X$ is a \emph{\catk\ space} if for any geodesic triangle $\triangle xy_1y_2$ of perimeter $< 2 D_K$, and its comparison triangle $\triangle \widetilde{x}\widetilde{y}_1\widetilde{y}_2$ in $M_K$, we have $$d(z_1,z_2) \leq d_K(\widetilde{z}_1,\widetilde{z}_2),$$ where $z_i$ is any point on $[x y_i]$ and $\widetilde{z}_i$ is the point on $[\widetilde{x} \widetilde{y}_i]$ such that $d_K(\widetilde{x},\widetilde{z}_i)=d(x,z_i)$ for $i \in \{1,2\}$.
\end{defn}

Let us define the (Alexandrov) \emph{angle} between two geodesics.
\begin{defn}
Let $\gamma_1$, $\gamma_2$ be two geodesics in $X$ starting at $x$.
The \emph{(Alexandrov) angle} $\angle_{x} (\gamma_1,\gamma_2)$ between $\gamma_1$ and $\gamma_2$ is given by $$\angle_{x} (\gamma_1,\gamma_2) := \limsup_{t_1,t_2 \to 0}  \theta(t_1,t_2)$$ where $\theta(t_1,t_2)$ is the angle of $\triangle \widetilde{x} \widetilde{y}_1 \widetilde{y}_2$ at $\widetilde{x}$ for $y_i = \gamma_i(t_i)$.
\end{defn}
Note that we can get $\theta(t_1,t_2)$ with the law of cosines.
If $X$ is a \catk\ space, then $\theta(t_1,t_2)$ is non-increasing in both variables. So there exists $\lim_{t \to 0} \theta(t,t)$ and it is equal to $\angle_{x} (\gamma_1,\gamma_2)$. For $\triangle xyz$ in $X$, the \emph{angle} $\angle yxz$ of $\triangle xyz$ at $x$ is the Alexandrov angle between $[x y]$ and $[x z]$.

We need Helly's Theorem for general \catzero\ spaces.
\begin{theorem}[Helly's Theorem]\label{t111}\cite[Prop. 5.2]{ls} (see also \cite[B.1]{akp})
Let $X$ be a \catzero\ space and $\{Y_a \}_{a \in I}$ be an arbitrary collection of
	closed bounded convex subsets of $X$.

	If every finite index array
	$\{a_1, a_2, \ldots, a_n \} \subset I$ satisfies
	$$\bigcap_{i=1}^n Y_{a_i} \neq \emptyset,$$
  then
$$\bigcap_{a \in I} Y_a \neq \emptyset.$$
	\end{theorem}

\subsection{Tangent spaces}
We can generate a definition of a direction since the triangle inequality holds for angles between three geodesics and $\angle_x (\gamma,\gamma) = 0$ where $\gamma$ is a geodesic starting at $x \in X$.
Two geodesics $\gamma_1$ and $\gamma_2$ starting at $x \in X$  have the \emph{same direction at} $x$ if $\angle_x (\gamma_1,\gamma_2) = 0$ and we denote this relation by $\gamma_1 \sim \gamma_2$. This is an equivalence relation on the set of geodesics starting at $x$.
Then this set of equivalence classes is a metric space with metric $\angle_x$. Denote the equivalence class of $\gamma$ by $[\gamma]$. Now consider the intrinsic metric $d$ induced from $\angle_x$. Note that if $d( [\gamma_1],[\gamma_2] ) \leq \pi$, then
$d( [\gamma_1],[\gamma_2] ) = \angle_x ( [\gamma_1],[\gamma_2] )$.
The completion of this space with metric $d$ is called the \emph{space of directions at} $x$ and is denoted by $\Sigma_x$.

The Euclidean cone over $\Sigma_x$ is called the \emph{tangent cone} $T_x$ at $x$; the elements of $T_x$ are pairs $v=(\xi, r)$ where
       $\xi \in \Sigma_x$, $r \ge 0$ is a real number. We call $\xi$ the direction of $v$, and $r$ the length of $v$.
       All the pairs $(\xi, 0)$ are identified as $o_x$
       and $o_x$ is called the \emph{vertex} of $T_x$.
       The \emph{norm} on $T_x$ is given by $r = ||(\xi,r)||$,
       that is, it is the distance from the vertex $o_x$,
       and the angle between $(\xi, r), (\eta, s) \in T_x$,
       when both $r, s \ne 0$, is the same as the angle between
       $\xi, \eta$.

The \emph{inner product} $\langle v ,w \rangle$ for $v,w \in T_x$ is defined by $||v|| ||w|| \cos \theta$ where $\theta$ is the angle between $v$ and $w$ if $v \neq o_x$ and $w \neq o_x$. Otherwise, define $\langle v ,w \rangle = 0$.


\begin{theorem}\cite{n}
If $X$ is a \catk\ space, then $\Sigma_x$ is a \catone\ space, and $T_x$ is a \catzero\ space.
\end{theorem}

\begin{defn}
Let $X$ be a \catk\ space, and $\gamma :I \to X$ be a rectifiable curve. For $t,t' \in I$ such that $t \leq t'$, let $\xi_{t'}$ be the direction at $\gamma(t)$ of $[\gamma(t) \gamma(t')]$. The curve
$\gamma$ has a \emph{right-side tangent vector} $\gamma'(t+) = (\xi ,r )$ at $\gamma(t)$ if there exist $$r=\lim_{t' \to t+} { d(\gamma(t'),\gamma(t)) \over t'-t },$$ and $$\xi=\lim_{t' \to t+} \xi_{t'}.$$ 
\end{defn}

Let us see the First Variation Formula for \catk\ spaces.
\begin{theorem}\cite[Page 185]{bh}
Let $X$ be a \catk\ space. For a distance $f(t)$ between unit-speed geodesics $\gamma_1(t)$ and $\gamma_2(t)$,
if $f(0) < D_K$, then $$  f'(0) = - \cos \alpha_1  - \cos \alpha_2 $$
where $\alpha_i$ is the angle at $\gamma_i(0)$ between the geodesic $\gamma_i$ and the geodesic $[\gamma_1(0) \gamma_2(0)]$.
\end{theorem}

\section{Semi-convex functions and their gradient vectors}
Suppose $(X,d)$ is a \catzero\ space.
For $x_0,x_1 \in X$ and $0 \leq t \leq 1$, let $x_t$ be the point on $[x_0 x_1]$ such that $t d(x_0, x_1)=d(x_0,x_t)$ and $(1-t) d(x_0, x_1)=d(x_1,x_t)$.
\begin{defn}\label{a6}
For $\lambda \in \mathbb{R}$, a function $F: X \to \mathbb{R}$ is \emph{$\lambda$-convex} if
\begin{equation}\label{eee}
 F ( x_t )  \leq (1-t) F(x_0) + t F(x_1) - t(1-t) {\lambda d^2(x_0,x_1) \over 2}.
\end{equation}
for any $x_0$, $x_1 \in X$.
\end{defn}

\begin{prop}\label{g1}
A geodesic metric space $X$ is \catzero\ if and only if
for any $y$, $x_0$ and $x_1 \in X$, $$d^2(y,x_t) \leq (1-t)d^2(y,x_0) + t d^2(y,x_1) - t(1-t)d^2(x_0,x_1).$$
Briefly, if and only if functions $x \mapsto d^2(y,x)$ are 2-convex.
\end{prop}

Using Theorem \ref{t111}, we have
\begin{lemma}\cite[Lemma 1.3]{m}\label{l111111}
Let $X$ be a \catzero\ space and $F: X \to \mathbb{R}$. If $F$ is convex and lower semi-continuous, then $F$ is bounded from below on bounded subsets of $X$.
Furthermore, the infimum of $F$ on each nonempty bounded convex closed subset of $X$ is attained.
\end{lemma}

To define gradient vectors of $\lambda$-convex functions, we need \emph{differentials} of $\lambda$-convex functions.

\begin{defn}
For $x \in X$ and a locally Lipschitz function $F: X \to \mathbb{R}$, a function $d_x F: T_x \to \mathbb{R}$ is called the \emph{differential} of $F$ at $x$ if for any curve $\gamma$ such that $\gamma(0)=x$ and $\gamma'(0+)$ is defined, $$d_x F( \gamma'(0+) ) = \lim_{t \to 0+} {F \circ \gamma(t) - F \circ \gamma(0) \over t}.$$
\end{defn}

In \cite[Lemma 2.4]{k2}, Kleiner showed if a $\lambda$-convex function $F : X \to \mathbb{R}$ is $L$-Lipschitz on a \catzero\ space $X$, then for every $x \in X$, there is a unique $L$-Lipschitz function $d_x F : T_x  \to \mathbb{R}$. Moreover, $d_x F$ is convex and homogeneous of degree 1.

Then we can define \emph{gradient vectors} of $\lambda$-convex functions.

\begin{defn}\label{gg2}
A tangent vector $v \in T_x $ is called the \emph{downward gradient vector} of $F$ at $x$ if
\begin{enumerate}
\item $(d_x F)(w) \geq -  \langle v,w \rangle$ for all $w \in T_x $, and
\item $(d_x F)(v) = - \langle v,v \rangle$.
\end{enumerate}
We denote $v$ by $\nabla_x(-F)$.
\end{defn}

So the geometric meaning of the (downward) gradient vector is that $F$ will be decreased fastest in the direction of this gradient and the length of the gradient vector is the rate at which $F$ decreases in that direction.

\begin{theorem}\label{g11}
Let $X$ be a \catzero\ space.
If $F$ is locally Lipschitz and $\lambda$-convex on $X$, then for any point $x \in X$, there is a unique downward gradient vector $\nabla_x (-F) \in T_x$.
\end{theorem}
\begin{proof}
For uniqueness, if $v, v^\prime$ are two distinct downward gradient vectors of $F$ at $x$, then
$$|| v ||^2 = -(d_x F)(v) \leq \langle v ,v^\prime \rangle,$$   $$|| v' ||^2 = -(d_x F)(v^\prime) \leq \langle v ,v^\prime \rangle.$$
Then these inequalities imply that $|| v ||=0$ if and only if $\langle v ,v' \rangle=0$, hence if and only if $||v'||=0$ by the inner product definition. It follows that $v= v' = o_x$.
Otherwise if $||v|| >0$ and $||v'|| >0$, by the inner product definition, we have $$ || v ||^2\leq || v || || v' || \cos \theta, \; ||v'||^2 \leq ||v|||| v' || \cos \theta,$$ where $\theta$ is the angle between $v$ and $v'$. Therefore $$1 \leq \cos^2\theta$$ since $|| v || \leq || v' || \cos \theta \leq ||v|| \cos^2 \theta.$
Since $\cos \theta > 0$ because $0 < || v || \leq || v' || \cos \theta$, we obtain $\cos \theta = 1$ and $\theta =0$. Thus $v=v'$.

For existence, first if $d_x F \geq 0$ then $\nabla_x(-F)$ is defined to be $o_x$. Otherwise, let $$r=\inf_{\eta \in \Sigma_x } (d_x F)(\eta) < 0$$ where $\Sigma_x$ is the direction space at $x$. Let $S_x$ be the unit ball $ \{ w \in T_x  | \| w \| \leq 1 \}$ of the \catzero\ space $T_x$.
Since $d_x F$ is Lipschitz and convex on $T_x$, $d_x F$ attains its infimum on the nonempty bounded convex closed subset $S_x$ by Lemma \ref{l111111}. Since $d_x F$ is homogeneous, $\inf_{S_x} d_x F = r$. So we have a minimum direction $\xi$ such that $ d_x F ( \xi ) = r$.
Then $v= ( \xi , |r| )$ satisfies the definition of the downward gradient vector, as follows:

\begin{enumerate}
\item When $\xi$ is the minimum point of $d_x F$ on the closed ball
$S_x$, the convexity of $d_x F$ gives the support inequality
$$d_x F(\eta) \ge d_x F( \xi)\cos(s) = r\langle \xi, \eta \rangle,$$
where $\eta \in \Sigma_x$ and $s = d_{\Sigma_x}( \xi, \eta ) < \pi$.

From the support inequality the proof of defining property (1)
for the gradient vector $v$ easily follows from the
homogeneity of $d_x F$ and $\langle \cdot, \cdot \rangle$:

For $\eta \in \Sigma_x$ and $s < \pi$,
$$d_x F(\eta) \ge r\langle \xi, \eta \rangle = - |r|\langle \xi, \eta \rangle
= -\langle v, \eta \rangle.$$

When $d_{\Sigma_x}( \xi, \eta) = \pi$, then the geodesic from
$\xi$ to $\eta$ goes through the origin $o_x$ and the inequality we want
is $d_x F(\eta) + d_x F( \xi) \ge 0$, as follows from the convexity of $d_x F$
on that geodesic.
\item $d_x F ( v ) = |r| d_x F ( \xi ) = |r|r = - |r|^2 \langle \xi , \xi \rangle = - \langle v , v \rangle$.
\end{enumerate}
\end{proof}

We have an important lemma about gradient vectors at different points.
\begin{lemma}\cite{l}\cite[Lemma 1.3.3]{p}\label{l222}
Let $X$ be a \catzero\ space and $F$ be locally Lipschitz and $\lambda$-convex. Then for any two different points $x$ and $y$,
$$  \langle \xi_1 , \nabla_{x} (- F) \rangle + \langle \xi_2 , \nabla_{y} (- F) \rangle   \geq  \lambda d(x,y) $$
where $\xi_1$ is the direction of $[xy]$ at $x$ and $\xi_2$ is the direction of $[yx]$ at $y$.
\end{lemma}
\begin{proof}
By definition, we have $$ \langle \xi_1 , \nabla_{x} (- F) \rangle \geq - (d_x F)( \xi_1 ) = - \lim_{t \to 0+} {F \circ \gamma(t) - F \circ \gamma(0) \over t} $$ where $\gamma = [xy]$.
Since $F$ is $\lambda$-convex, we get $$\langle \xi_1 , \nabla_{x} (- F) \rangle \geq {F(x) - F( \gamma( d(x,y) ) ) +  \lambda d(x,y)^2 /2 \over d(x,y) }.$$
Doing similarly for $\langle \xi_2 , \nabla_{y} (- F) \rangle$ and adding two inequalities, it is proved.
\end{proof}

\begin{defn}\label{d11111}
Let a function $F: X \to \mathbb{R}$ be locally Lipschitz and $\lambda$-convex on $X$. A locally Lipschitz curve $\gamma : I \to X$ is a \emph{gradient curve} of $F$ if for all $t \in I$,
there exists the right-side tangent vector $\gamma^\prime(t+)$ and it is equal to the downward gradient vector $\nabla_{\gamma(t)} (- F)$ at $\gamma(t)$.
\end{defn}

\section{Distance between two fixed-time gradient curves}
First, we will see Mayer's work showing existence and uniqueness of time-independent gradient curves. Then
we are going to look at two results about the distance between two fixed-time gradient curves: Corollary \ref{g77777} and Theorem \ref{g21}. These results will be used to prove our main Theorem \ref{g7}. To emphasize the two results, we name them Distance I and Distance II.

In this section, we assume that $(X,d)$ is a \catzero\ space. For a function $F=F(t,x)$ on $\real \times X$, let $F_t$ be given by $F_t(x):=F(t,x)$. Then
$F$ is $\lambda$-\emph{convex} on $X$ if the function $F_t$ is $\lambda$-convex.

In order to get Distance II, we will need the H\"{o}lder continuity of gradient vectors $\nabla_{x} ( - F_{t})$ in the time variable $t$.
This need is illustrated by Example \ref{ex111} below.

Distance I is essentially a result of Mayer about time-independent gradient curves. First, we need to define a \emph{step-energy function} of $F$.

\begin{defn}\cite[Def. 3]{m}\label{g2}
Given an initial position $x_0 \in X$, an initial time $t_0 \in \mathbb{R}$ and a time gap $h > 0$, the \emph{step-energy function} $E_{t_0,x_0,h} : X \to \mathbb{R}$ at $(t_0,x_0)$ is defined by
$$x \mapsto F(t_0,x)+ {1 \over 2h} d^2(x_0,x).$$
\end{defn}

Since for sufficiently small $h$, $E_{t_0,x_0,h}$ is a convex function with bounded sublevel sets, we can get a minimum value of $E$.
Differently from Mayer's setting, we need to consider time $t_0$ as variable.
Mayer does not use the term ``$\lambda$-convex'' and he defines his condition on the function $F$ in terms of the parameter $S = -\lambda /2$. He also restricts to the case $S > 0$, while we do not make any restriction on the sign of $\lambda$.
\begin{prop}\cite[Th. 1.8]{m}\label{g3}
Suppose that $F$ is $\lambda$-convex and locally Lipschitz on a \catzero\ space $X$. Let $h>0$ and if $\lambda < 0$, let $h < -1/2 \lambda$.
Then $E_{t_0,x_0,h}$ has a unique minimum point on $X$.
\end{prop}
\begin{defn}
We will denote the unique minimum point of the step-energy function $E_{t_0,x_0,h}$ by $$J(t_0,x_0,h).$$
\end{defn}
The function from $X$ to $X$ given by $x \mapsto J(t_0,x,h)$, which we will call the \emph{discrete flow function} with time gap $h$, was studied in detail by Mayer.
We have the following result, extending \cite[Lemma 1.12]{m} to $\lambda \ge 0$, a longer interval for $h$, and giving a slightly smaller Lipschitz constant instead of $1 / \sqrt{ 1 +  2 \lambda h}$.
\begin{lemma}\cite[Lemma 1.12]{m}\label{g6}
Let $h$ be as in Proposition \ref{g3}. Then
for any $t_0 \in \mathbb{R}$, the
discrete flow
function $x \mapsto J(t_0,x,h)$ is $ $ $(1 + \lambda h)^{-1}$-Lipschitz.
\end{lemma}
This lemma is one of the primary tools to show that discrete flow converges well when $h$ goes to zero.

For any $y \in X$, let
\begin{equation}\label{e222}
\begin{split}
A(t_0) &= - \min \{0, \liminf_{ d(x,y) \to \infty}  {F(t_0,x) \over d^2(x,y)}  \}, \\
I_{A(t_0)} &=  \left\{ \begin{array}{ll}
 (0,\infty)& \mbox{ if $A(t_0) = 0$ },  \\
 (0, {1 \over 16 A(t_0) }]& \mbox{ if $A(t_0) > 0$ }.
\end{array}
\right.
\end{split}
\end{equation}
Note that $A(t_0)$ is independent of $y$ because of the triangle inequality.

Mayer obtains only weak gradient curves in his main theorem since he assume the weaker condition that $F$ is semi-continuous. Tangent vectors and gradient vectors were not contained in the definition of the weak gradient curve. If we assume $F$ is locally Lipschitz on $X$, by Theorem \ref{g11}, we have gradient vectors everywhere and Mayer's result can be modified for a gradient curve as follows (see Definition \ref{d11111}).
We will see this modification in Section 7.
\begin{theorem}[Existence and Uniqueness of time-independent gradient curves]\cite[Th. 1.13]{m}\label{g777}
Let $X$ be a \catzero\ space.
Suppose that $F$ is $\lambda$-convex and locally Lipschitz on $X$.
For an initial position $x_0$ and fixed-time $t_0$, let $x^n_0(s)= x_0$ and
$$x^n_i(s)=J(t_0,x^n_{i-1}(s),s/n)$$ where $i=1, \cdots, n$. Then there is a unique gradient curve $\mu: I_{A(t_0)} \to X$ of the function $F(t_0, \cdot \; )$ defined by $\mu(s) := \lim_{ n \to \infty} x^n_n(s)$ such that $\mu(0)=x_0$.
\end{theorem}
Note that for simple notation, we use $x^n_i(s)$ even though it is also dependent on $x_0$. For each $n$, there are $n+1$ points of $x^n_i(s)$'s and $\lim_{ n \to \infty} x^n_n(s)$ is the limit of
the sequence $\{x^n_n(s) | n \in \mathbb{N} \}$ where $x^n_n(s)$ is the last point of $x^n_i(s)$'s.
\begin{cor}[Distance I]\cite[Th. 2.1]{m}\label{g77777}
Let $\mu_y$ be the gradient curve of the function $x \mapsto F(t_0,x)$ where $\mu_y(0) =  y$. Then $$d( \mu_{x_1}(s) , \mu_{x_2}(s) ) \leq e^{- \lambda s } d( x_1, x_2 ).$$
\end{cor}

In order to obtain Distance II, we are going to assume that the
function $t \mapsto \nabla_{x} ( - F_{t})$ is H\"{o}lder continuous.
To see why we will need such a condition when we turn to time-dependent gradient curves, let us look at the following examples.

\begin{defn}\label{de99}
Let $X$ be a \catzero\ space and $F: \real \times X \to \real$ be $\lambda$-convex. A locally Lipschitz curve $\sigma : [t_0,T] \to X$ is a \emph{time-dependent gradient curve of $F$ at $x_0$ and $t_0$} if $\sigma(t_0)=x_0$, there exists the right-side tangent vector $\sigma^\prime(t+)$ for all $t \in [t_0,T)$ and it is equal to the downward gradient vector $\nabla_{\sigma(t)} (- F_{ t })$ at $\sigma(t)$.
\end{defn}
We start with a time-independent example.

\begin{examp}\label{ex1111}
Let $X$ be the subset of Euclidean plane such that $x \geq 0$ and $y \geq 0$.
Let $F(x,y):= - \min \{ x,y \}$. Then the gradient vector $\nabla_{(x,y)} (- F)$ is $(1,0)$ if $x<y$, $(0,1)$ if $y<x$ or $(1/2,1/2)$ if $x=y$.
Because $X$ is a manifold with boundary, we can put the tangent bundle metric on the set of all tangent vectors at all points.
Here $(x, y) \mapsto \nabla_{(x,y)}(-F)$ has discontinuities at the points where $y = x$ because its length is $1/\sqrt{2}$ at those points and length $1$ everywhere else.

Thus for any initial point $(x_0,0)$ on $X$, we have a gradient curve $\gamma$ of the function $F$ given by
$$
\gamma(s) =  \left\{ \begin{array}{ll}
 (x_0,s) & \mbox{ if $0 \leq s \leq x_0$ },  \\
 ({1 \over 2}(x_0+s) ,{1 \over 2}(x_0+s)) & \mbox{ if $s \geq x_0$}.
 \end{array}
\right.
$$
\end{examp}

\begin{figure}[h]
\centering{\epsfig{file=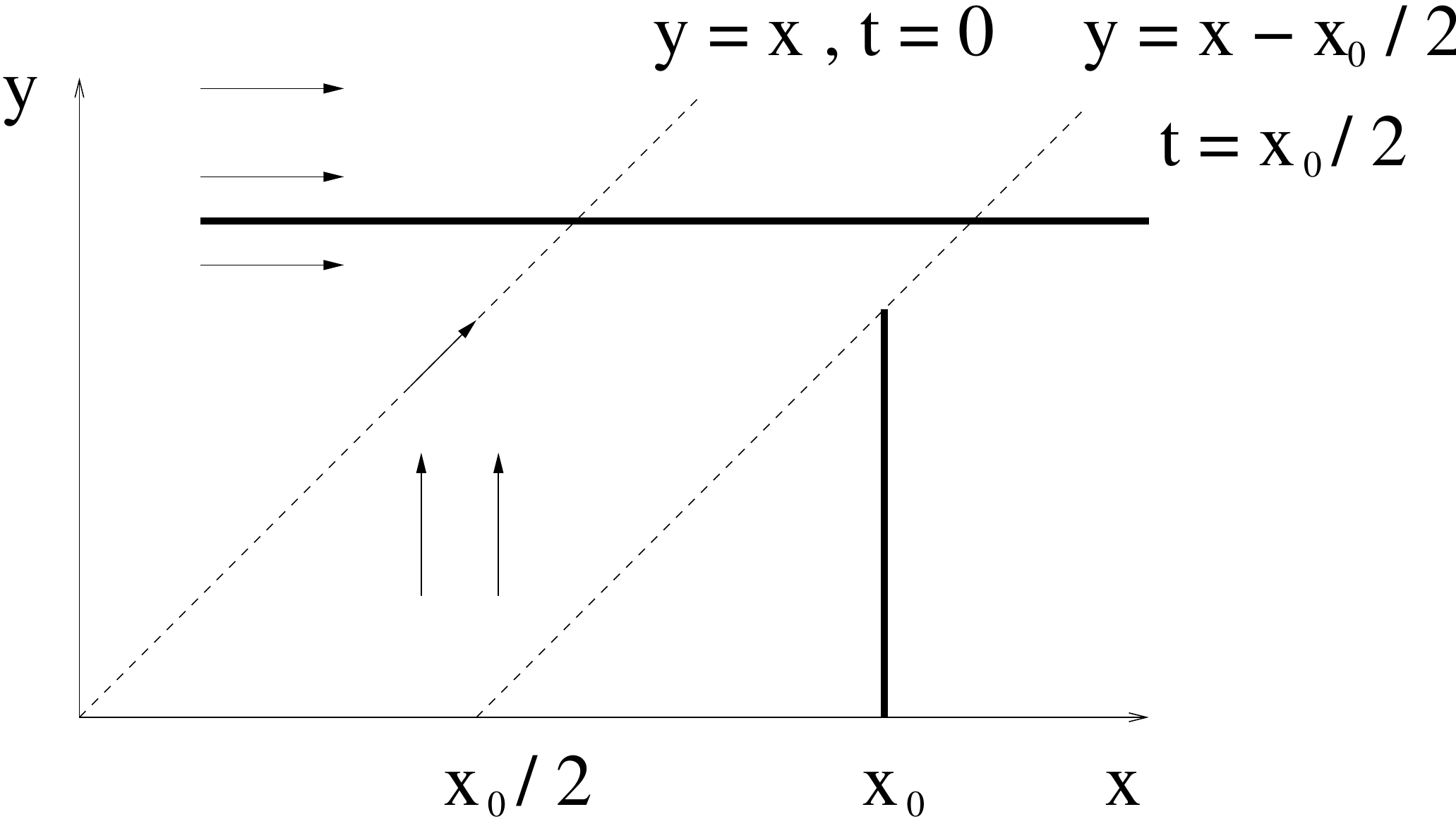  , height=4.8cm, width=9cm}}
\caption{Gradient vectors of $F_t$ at time $t=0$ and two time-dependent gradient curves}\label{forty}
\end{figure}

Now we give a time-dependent example of a convex function having no time-dependent gradient curves at some points.
In this example, the singular locus of Example \ref{ex1111} is translated $t$ units to the right for each $t \geq 0$.

\begin{examp}\label{ex111}
Let $X$ be the subset of Euclidean plane such that $x \geq 0$ and $y \geq 0$.
For $t \geq 0$, let $F(t,x,y):= - \min \{ x-t ,y \}$.
Then $(t, x, y) \mapsto \nabla_{(x,y)}(-F_t)$ has discontinuities at the points where $y = x-t$.
If a time-dependent gradient curve leads to one of these points, it terminates and cannot be continued as a gradient curve. A gradient curve starting above the diagonal $y=x$ never reaches one of these points and hence is defined for all $t$.
Those starting below the diagonal terminate in finite time when they get half way to the diagonal. No gradient curve can start at the diagonal. See Figure \ref{forty}.
\end{examp}

\begin{lemma}\label{l1}
Let $(X,d)$ be a \catk\ space and
$(T_x ,\rho)$ be the tangent cone of $X$ at $x \in X$.  Suppose $v_1, v_2, v_3 \in T_x$
and $\| v_3 \| = 1$. Then
$$|\langle v_3, v_1 \rangle - \langle v_3, v_2 \rangle | \le \rho(v_1, v_2).$$
\end{lemma}
\begin{proof}
Let $\theta_{ij} = \angle (v_i,v_j)$. Without loss of generality, assume
$\theta_{13} \le \theta_{23}$.

In the Euclidean upper half-plane, set
$$
w_3 = (1,0), \quad
w_i =  \|v_i\|(\cos \theta_{i3},\sin \theta_{i3}),\ i=1,2.
$$
Since projection to the $x$-axis does not increase distance,
$$|\langle v_3, v_1 \rangle - \langle v_3, v_2 \rangle|
= \bigl|  \|v_1\|\cos \theta_{13} - \|v_2\|\cos \theta_{23}\bigr|
\le \|w_1 - w_2\|.$$
By the triangle  inequality for angles,
$\theta_{23}-\theta_{13} \le \theta_{12}$. If $\theta_{12}<\pi$, then
$$ \|w_1 - w_2\|\le \rho(v_1, v_2),$$
since the righthand side may be obtained from the lefthand side by increasing the hinge angle
$\angle (w_1,w_2)$ from $\theta_{23}-\theta_{13}$ to $\theta_{12}$.   If $\theta_{12}=\pi$, then $\rho(v_1, v_2)=\|w_1\|+\|w_2\|$ and the inequality still holds.
\end{proof}

After we prove our main Theorem \ref{g7}, we will give several important examples in Section 6 that satisfy the assumption of H\"{o}lder continuity.

To finish this section, let us see how this H\"{o}lder continuity works on two fixed-time gradient curves issuing from the same point.

\begin{theorem}[Distance II]\label{g21}   
Let $(X,d)$ be a \catzero\ space and $(T_x , \rho )$ be the tangent cone of $X$ at $x \in X$.
Given $x_0 \in X$ and $t_1,t_2 \in \mathbb{R}$, suppose a function $F:\mathbb{R} \times X \to \mathbb{R}$ satisfies \\
1) $F$ is locally Lipschitz on $X$, \\
2) $F$ is $\lambda$-convex on $X$, \\
3) $\exists$ $B>0$ and $\alpha>0$ such that $\rho \bigl( \nabla_{x} (-F_{t_1})  , \nabla_{x} (-F_{t_2}) \bigr) \leq B | t_1 - t_2 |^\alpha$ for any $x \in X$.\\
Let $\mu_i:I_{A(t_i)} \to X$ be the fixed-time $t_i$ gradient curve of the function $x \mapsto F(t_i,x)$ where $\mu_i(0) =  x_0$ for $i=1,2$.  Then there is a positive constant $T$ such that $$d ( \mu_1(s), \mu_2(s) ) \leq 2 B s |t_1-t_2|^\alpha $$ for all $s$ in $\overline{I_{A(t_1)}} \cap \overline{I_{A(t_2)}} \cap [0,T]$.
\end{theorem}
Note $I_{A(t_i)}$ is an interval dependent on $F_{t_i}$, given by Mayer (see Theorem \ref{g777}). $\overline{I_A}$ is a closure of $I_A$.
\begin{proof}
Let $f(s)$ be a distance $d( \mu_1(s), \mu_2(s) )$ and let $\xi_1(s)$ be the unit vector of $[\mu_1(s) \mu_2(s)]$ at $\mu_1(s)$  and $\xi_2(s)$ be the unit vector of $[\mu_2(s) \mu_1(s)]$ at $\mu_2(s)$.

By first variation formula, we have
$$ f^+(s) = - \langle \xi_1(s)  ,  \nabla_{ \mu_1(s) } (- F_{ t_1 }) \rangle - \langle  \xi_2(s) , \nabla_{ \mu_2(s) } (- F_{ t_2 }) \rangle.  $$
We have three vectors $\nabla_{ \mu_2(s) } (- F_{ t_1 })$, $\nabla_{ \mu_2(s) } (- F_{ t_2 })$ and $\xi_2(s)$ in a tangent cone $T_{ \mu_2(s) }$.
By Lemma \ref{l1} with three vectors of $T_{ \mu_2(s) }$, we get
$$ f^+(s) \leq - \langle \xi_1(s)  ,  \nabla_{ \mu_1(s) } (- F_{ t_1 }) \rangle - \langle  \xi_2(s) , \nabla_{ \mu_2(s) } (- F_{ t_1 }) \rangle +  \rho ( \nabla_{ \mu_2(s) } (- F_{ t_1 }), \nabla_{ \mu_2(s) } (- F_{ t_2 }) ).  $$
By Lemma \ref{l222} with $x=\mu_1(s)$ and $y=\mu_2(s)$, we have
$$ \langle \xi_1(s)  ,  \nabla_{ \mu_1(s) } (- F_{ t_1 }) \rangle + \langle  \xi_2(s) , \nabla_{ \mu_2(s) } (- F_{ t_1 }) \rangle   \geq  \lambda f(s).$$
Finally, we get
\begin{equation}
\begin{split}
f^+(s) & \leq - \lambda f(s) + \rho ( \nabla_{ \mu_2(s) } (- F_{ t_1 }), \nabla_{ \mu_2(s) } (- F_{ t_2 }) ) \\
& \leq - \lambda f(s) + B | t_1 - t_2 |^\alpha.
\end{split}
\end{equation}
Since $f(s)$ is locally Lipschitz, $f^+(s)$ exists for almost each $s$. Thus we need to solve this differential equation.
If $\lambda = 0$, it is proved trivially.
So we only need to consider the case $\lambda \neq 0$.
Solving this,
$$ d( \mu_1(s), \mu_2(s) ) = f(s) \leq B | t_1 - t_2 |^\alpha   ( 1  - e^{  - \lambda s }  )  /  \lambda $$ for all $s \geq 0$.
Note that $y = {n \over -m} + C e^{mx}$ is the solution of $y' = m y + n$.

If $\lambda \geq 0$, it is proved since $1 - e^{- \lambda t} \leq \lambda t$.

If $\lambda < 0$, let $\lambda_1 = - \lambda $. Then
\begin{equation*}
\begin{split}
(1 -  e^{ - \lambda s })  /  \lambda  &=  (1 - e^{  \lambda_1 s } ) / (- \lambda_1)  \\
&= ( e^{ \lambda_1 s } - 1 ) / \lambda_1   \\
\end{split}
\end{equation*}
for all $s \geq 0$. Then it is proved because there is an positive constant $T$ such that $e^{\lambda_1 s } - 1 \leq 2 \lambda_1 s$ for all $s$ in $[0,T]$.
\end{proof}

\section{Time-dependent gradient curves}
Now we will show the existence and uniqueness of time-dependent gradient curves (see Definition \ref{de99}).
For this, we need the definition of the flow map of the function $F_t$ when $F$ is Lipschitz in $t$.
\begin{defn}\label{de9999}
 Let $X$ be a \catzero\ space and $F$ be $\lambda$-convex and locally Lipschitz on $X$. Suppose that $F$ is $L$-Lipschitz in $t$. Fix $t$ and let $\mu_{x,t} : I_A \to X$ be the fixed-time $t$ gradient curve of the function $F_{t}$ with $\mu_{x,t}(0)=x$.
Since $F$ is $L$-Lipschitz in $t$, $I_A$ (defined in \eqref{e222}) is independent of $t$.
Then from Theorem \ref{g777}, the flow map $\Phi$ can be defined by
\begin{equation}
\begin{split}
\Phi : \mathbb{R} \times \overline{I_A} \times X     & \to X   \\
       (t \; ,\ell \; ,x \; )  & \mapsto \mu_{x,t}( \ell ) \\
\end{split}
\end{equation}
\end{defn}
For fixed $t$, Mayer showed the semigroup property of this flow map $\Phi$, which will be used in the proof of Theorem \ref{g7}.
\begin{theorem}\cite[Th. 2.5]{m}\label{g9999}
 For $\ell_1$, $\ell_2 \geq 0$ such that $\ell_1+\ell_2 \in I_A$, $$\Phi( t, \ell_1 + \ell_2 , x) = \Phi ( t, \ell_2 , \Phi (t, \ell_1 ,x)).$$
\end{theorem}

Next, we construct a set of piecewise fixed-time gradient curves beginning at $x_0$, and show that they converge a continuous curve.
Let $t^n_i(s)= t_0 + i s /2^n$ for fixed $s$ and an integer $i \geq 0$. Let $p^n_0(s) = x_0$ and
let the vertices of the piecewise fixed-time gradient curves be
$$p^n_i(s) = \Phi( t^n_{i-1}(s), s/2^n , p^n_{i-1}(s) ).$$ See Figure \ref{five}.
Note that $p^n_i(s)$ is dependent on $x_0$ and $t_0$. But we use $p^n_i(s)$ to avoid complicated notation.

We define \emph{piecewise fixed-time gradient curves} $\gamma^n_s : [t_0, \infty) \to X $ \emph{with step size $s/2^n$} such that $\gamma^n_s(t_0)= x_0$ given for $t^n_{i-1}(s) \leq t \leq t^n_i(s)$
by $$\gamma^n_s(t) := \Phi ( t^n_{i-1}(s), t - t^n_{i-1}(s) , p^n_{i-1}(s) ).$$

Two points $p^n_{i-1}(s)$ and $p^n_i(s)$ are connected by the fixed-time $t^n_{i-1}(s)$ gradient curve of $F_{t_{i-1}(s)}$. Then $p^n_i(s)$ is connected to $p^n_{i+1} (s)$ by the fixed-time $t^n_i(s)$ gradient curve of $F_{t_i(s)}$.
The fixed-time $t^n_{i}(s)$ gradient curve flows for time $s/2^n$ from $p^n_{i}(s)$ for each $i \geq 0$.
In Figure \ref{five}, there are two piecewise fixed-time gradient curves of $\gamma^n_s$ and $\gamma^{n+1}_s$. Note $t^n_{i}(s)=t^{n+1}_{2i}(s)$.

\begin{theorem}[Existence and Uniqueness of time-dependent gradient curves]\label{g7}
Let $(X,d)$ be a \catzero\ space and $(T_x ,\rho)$ be the tangent cone of $X$ at $x \in X$. Given $x_0 \in X$ and $t_0 \in \mathbb{R}$, suppose a function $F:\mathbb{R} \times X \to \mathbb{R}$ satisfies \\
1) $F$ is locally Lipschitz on $X$, \\
2) $F$ is $\lambda$-convex on $X$, \\
3) $F$ is $L$-Lipschitz in $t$, \\
4) $\exists$ $ B >0$, $B_0 > 0$ and $\alpha>0$ such that $\rho \bigl( \nabla_{x} (-F_{t})  , \nabla_{x} (-F_{t'}) \bigr) \leq B | t - t' |^\alpha$ for any $x \in X$, any $t, t' \in \mathbb{R}$ such that $|t-t'| \leq B_0$.

Then there is a time-dependent gradient curve $\sigma_{x_0,t_0}$ of the function $F$ at $x_0$ and $t_0$ defined by $\sigma_{x_0,t_0}(t_0+s) := \lim_{ n \to \infty} \gamma^n_s (t_0+s)$
such that $\sigma_{x_0,t_0}(t_0 ) = x_0$. Moreover, any time-dependent gradient curve of $F$ at $x_0$ and $t_0$ coincides with $\sigma_{x_0,t_0}$. 
\end{theorem}
Note that $\gamma^n_s(t_0+s) = p^n_{2^n}(s)$ since $p^n_{2^n}(s)$ is $(2^n+1)$-th point of $\{ p^n_{i}(s) \ | i \geq 0 \}$.

\begin{proof}
Since $F$ is $L$-Lipschitz in $t$, $I_A$ (defined in \eqref{e222}) is independent of $t$. For a large integer $n$, assume that $h = s / 2^n \in \overline{I_A} \cap [0,T] \cap [0,B_0] $ where $T$ is from Theorem \ref{g21}.



\begin{figure}[h]
\centering{\epsfig{file=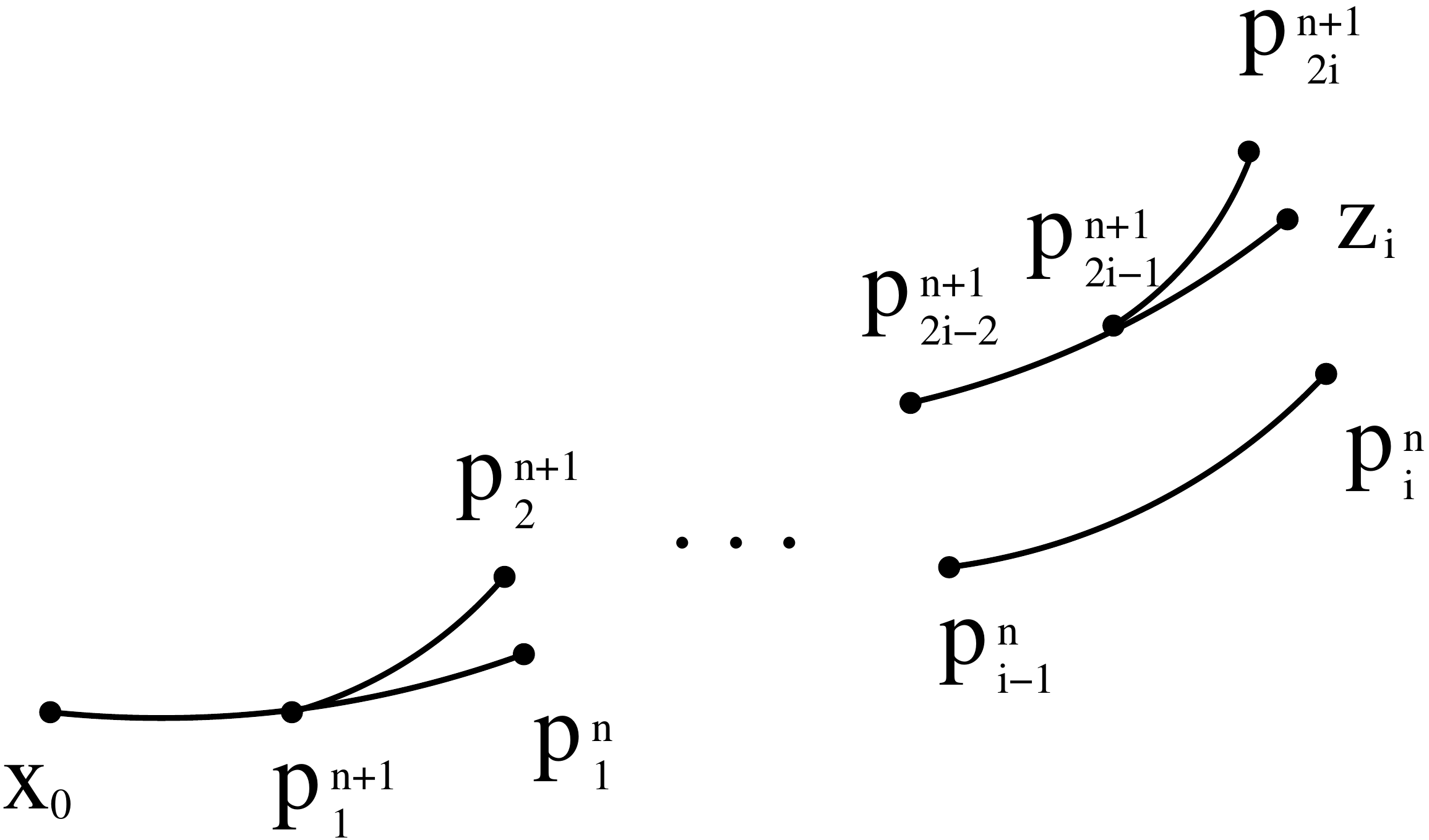  , height=5.3cm, width=9.6cm}}
\caption{ Two piecewise fixed-time gradient curves of $p^n$ and $p^{n+1}$ and the point $z_i$ }\label{five}
\end{figure}

\textbf{Claim 1.} \textit{The sequence} $\{p^n_{2^n}(s)   \}$ \textit{is Cauchy. We define the limit curve} $\sigma_{x_0,t_0}$ \textit{by} $$\sigma_{x_0,t_0}(t_0+s) := \lim_{n \to \infty} p^n_{2^n}(s).$$\\
By induction on $i$, we will show that $$d(p^n_i(s),p^{n+1}_{2i}(s) ) \leq Bi e^{ -\lambda_0 (i-1)  h } (h/2)^{1+\alpha}  $$ where $\lambda_0 = \min \{ 0,\lambda  \}$.

The start of the induction is trivial since $x_0 = p_0^n(s) = p_{2\cdot 0}^{n+1}(s)$.
First, assume that
\begin{equation}\label{ee1}
d(p^n_{i-1}(s),p^{n+1}_{2(i-1)}(s) ) \leq B(i-1) e^{ -\lambda_0 (i-2) h } (h/2)^{1+\alpha}.
\end{equation}

Let $z_{i}$ be the point on the fixed-time
$t_0+{(i-1)}h$ gradient curve flowing for time $h=s/2^n$ from $p^{n+1}_{2i-2}(s) $, i.e $$z_{i} := \Phi(  t^n_{i-1}(s),  s/2^n,  p^{n+1}_{2i-2}(s) )$$ (see Figure \ref{five}). Note also that $p_1^n(s) = z_1$.

Since  $z_{i} = \Phi(  t^n_{i-1}(s),  s/2^n,  p^{n+1}_{2i-2}(s) )$ and $p^n_i(s) = \Phi( t^n_{i-1}(s), s/2^n , p^n_{i-1}(s) )$, by Corollary \ref{g77777},
we get
\begin{equation}\label{ee2}
\begin{split}
d(p^n_{i}(s), p^{n+1}_{2i}(s) ) &\leq d(p^n_{i}(s), z_{i})+d(z_{i},p^{n+1}_{2i}(s) ) \\
& \leq e^{ -\lambda_0  h } d( p^n_{i-1}(s), p^{n+1}_{2(i-1)}(s) ) + d(z_{i},p^{n+1}_{2i}(s) ). \\
\end{split}
\end{equation}
By Theorem \ref{g9999}, the semigroup property of the fixed-time gradient
flows gives that $z_i$ is equal to the point flowing for time $h/2$ from
$p^{n+1}_{2i-1}(s)$ on the same fixed-time
$t_0+{(i-1)}h$ gradient curve, that is,
$$z_i = \Phi( t^n_{i-1}(s),  h/2 , p^{n+1}_{2i-1}(s) ).$$

Since $z_i = \Phi( t^n_{i-1}(s),  h/2 , p^{n+1}_{2i-1}(s) )$ and $p^{n+1}_{2i}(s) = \Phi( t^{n+1}_{2i-1}(s), h/2 , p^{n+1}_{2i-1}(s) )$, we can apply Theorem \ref{g21} at $p^{n+1}_{2i-1}(s)$ with $h/2 \leq T$.
Thus we have
\begin{equation}\label{ee3}
d(z_{i},p^{n+1}_{2i}(s) ) \leq 2 B (h/2) | t^{n+1}_{2i-1} - t^n_{i-1} |^{\alpha} = 2 B (h/2)^{1+\alpha}.
\end{equation}
By Equations \eqref{ee1}, \eqref{ee2} and \eqref{ee3}, we obtain
\begin{equation*}
\begin{split}
d(p^n_{i}(s) ,p^{n+1}_{2i}(s) )
& \leq e^{ -\lambda_0  h } 2 B(i-1) e^{ -\lambda_0 (i-2)  h } (h/2)^{1+\alpha}   + 2 B (h/2)^{1+\alpha}   \\
& = 2 B [ (i-1) e^{ -\lambda_0 (i-1)  h } + 1] (h/2)^{1+\alpha}  \\
& \leq 2 B i e^{ -\lambda_0 (i-1)  h } (h/2)^{1+\alpha}. \\
\end{split}
\end{equation*}
Letting $i$ be ${2^n}$, we have
\begin{equation}\label{g99}
d(p^n_{2^n}(s) , p^{n+1}_{2^{n+1}}(s) ) \leq  B s e^{ -\lambda_0 s} ( {\tfrac{s}{2^{n+1}} } )^\alpha.
\end{equation}
For $m$ and $n$ such that $m > n$, from Equation \eqref{g99},
\begin{equation}\label{ggg9ggg9}
\begin{split}
d(p^{n}_{2^{n}}(s) , p^m_{2^m}(s) ) &\leq d(p^{n}_{2^{n}}(s) , p^{n+1}_{2^{n+1}}(s) ) + \cdots + d(p^{m-1}_{2^{m-1}}(s) , p^m_{2^m}(s) )\\
&\leq B s^{1+\alpha}  e^{ -\lambda_0 s}  [    ( {\tfrac{1}{2^{n+1}} } )^\alpha      + \cdots +       ( {\tfrac{1}{2^{m}} } )^\alpha         ]  \\
&\leq B s^{1+\alpha}  e^{ -\lambda_0 s}  (\tfrac{1}{2^{ n }})^\alpha \tfrac{1}{1- 1/ 2^\alpha}. \\
\end{split}
\end{equation}
This means that we have the Cauchy sequence $\{p^n_{2^n}(s)  \}$.

\textbf{Claim 2.}
$$d( \sigma_{x_0,t_0}( t_0+s ) , p^n_{2^n}(s) ) \leq B' s^{1+\alpha} \tfrac{1}{2^{\alpha n}} e^{ -\lambda_0 s}$$ \textit{where} $\lambda_0 = \min \{ 0,\lambda  \}$
and $B'=  \tfrac{B}{1- 1/ 2^\alpha}$.

As $m \to \infty$ in Equation \eqref{ggg9ggg9}, we have
$$d(\sigma_{x_0,t_0}( t_0+ s) , p^n_{2^n}(s)  ) \leq B' s^{1+\alpha} \tfrac{1}{2^{\alpha n}} e^{ -\lambda_0 s}.$$
$$ $$
In next claim, for each fixed integer $n$ and any number $s_1$ smaller than $s$,
 we will deal with the piecewise fixed-time gradient curves $\gamma^n_{s_1}$ with step size $s_1/2^n$, which is less than step size $s/2^n$ of $\gamma^n_s$. Note that $\gamma^n_s$ is defined in the beginning of this section.

\textbf{Claim 3.}
\textit{For any fixed values} $s$, $s_1$ \textit{and} $t_0$, \textit{the two sequences of
points} $ \{ \gamma^n_s (t_0+s) | n \geq 0      \}$ \textit{and} $ \{  \gamma^n_{s_1} (t_0+s) | n \geq 0 \} $ \textit{get close in the sense that}
  $$\lim_{n \to \infty} d( \gamma^n_s(t_0+ s) , \gamma^n_{s_1}(t_0+ s) ) =0.$$

$$ $$
Let $h_1:=s_1/2^n$ and $h:=s/2^n$.
Let $s$ and $s_1$ be fixed and let $n$ be fixed until we take $n \to \infty$.
When $s_1=s/2^m$, it is proved by Claim 2.

 Let $x_{\ell}(s_1) = \gamma^n_{s_1} ( t_0 + {\ell} h_1 )$ for an integer $\ell \geq 0$. So $x_\ell(s_1)$ is the vertex of the piecewise fixed-time gradient curves $\gamma^n_{s_1}$ and
$$x_\ell(s_1) = \Phi( t_0 + (\ell-1) h_1,  h_1 , x_{\ell-1}(s_1) ).$$

Note that $t^n_i(s) = t_0 + i  s / 2^n $ for $i = 1, \cdots, 2^n.$
Since $s_1 < s$, there is a constant $k=k(i)$ such that $$t_0 + k h_1 \leq t^n_{i-1}(s) < t_0 + ( k + 1 )h_1.$$
(see Figure \ref{six}).

Note that if $s_1=s/2^{n+1}$, then Figure \ref{six} becomes Figure \ref{five}.
We want to get an upper bound of the distance between two points $\gamma^n_s(t^n_i(s))$ and $\gamma^n_{s_1}(t^n_i(s)) $. Here $\gamma^n_s(t^n_i(s))$ is the vertex of $\gamma^n_s$. But $\gamma^n_{s_1}(t^n_i(s))$ is not necessarily the vertex of $\gamma^n_{s_1}$ and there are several vertices x's between $\gamma^n_{s_1}(t^n_i(s) )$ and $\gamma^n_{s_1} (t^n_{i-1}(s) )$.
Specifically, by induction on $i$, we will show that
\begin{equation}\label{eeeee}
d(\gamma^n_s (t^n_{i}(s) ), \gamma^n_{s_1} (t^n_{i}(s) ) ) \leq 2 B( [s/s_1] +1) i s /2^n  e^{ -\lambda_0 i s /2^n }  ( s_1/ 2^n)^{\alpha} .
\end{equation}


\begin{figure}[h]
\centering{\epsfig{file=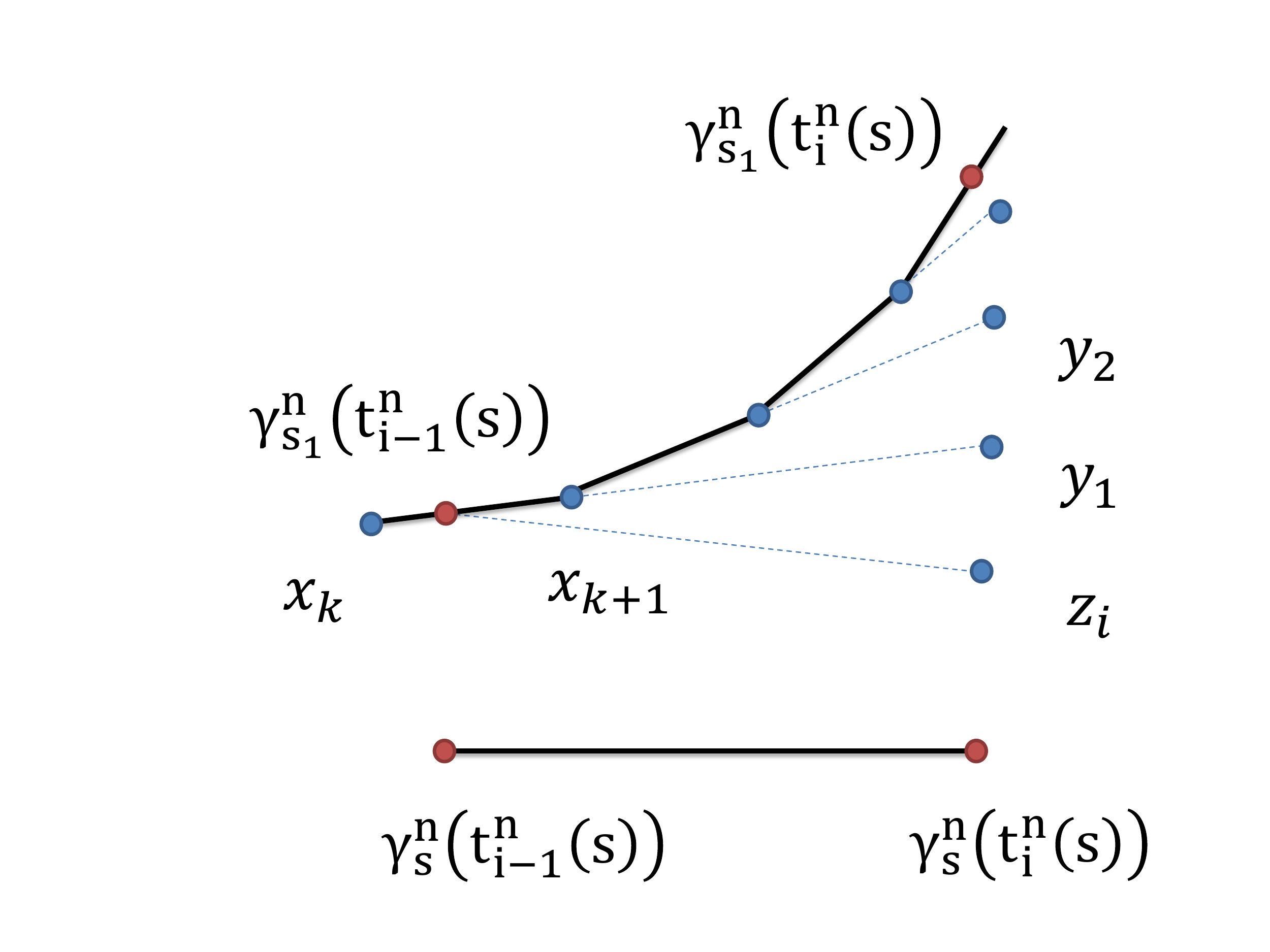  , height=7cm, width=9.7cm}}
\caption{ $\gamma^n_s$, $\gamma^n_{s_1}$, $y_j$ and $z_i$ }\label{six}
\end{figure}

In order to get the upper bound of the distance $d( \gamma^n_s (t^n_i(s) ) , \gamma^n_{s_1} (t^n_i(s) ) )$,
we will consider two distances $d( \gamma^n_s (t^n_i(s) ) , z_i)$ and $d( \gamma^n_{s_1} (t^n_i(s) ) , z_i)$ where $z_i$ will be defined below (see Figure \ref{six}).

\textbf{Step 1.} For $d( \gamma^n_s( t^n_i(s) ) , z_i)$, first look at the point $\gamma^n_{s_1}( t^n_{i-1}(s) )$. Let $t' = t_0 + k h_1$.
Then the point $\gamma^n_{s_1}( t^n_{i-1}(s) )$ is on the fixed-time $t'$ gradient segment from $\gamma^n_{s_1}( t' )$ to $\gamma^n_{s_1}( t' + h_1  ) = \Phi( t'  , h_1 , \gamma^n_{s_1}( t' ) )$.

Since the gradient segment from $\gamma^n_s(t^n_{i-1}(s) )$ to $\gamma^n_s(t^n_i(s) ) = \Phi( t^n_{i-1}(s) , s/2^n , \gamma^n_s(t^n_{i-1}(s) ) )$ has the fixed-time $t^n_{i-1}(s)$,
in order to use Corollary \ref{g77777}, we need a fixed-time $t^n_{i-1}(s)$ gradient curve from $\gamma^n_{s_1}(t^n_{i-1}(s))$. So let $z_{i}$ be the point on this fixed-time
$t^n_{i-1}(s)$ gradient curve flowing for time $s/2^n$ from $\gamma^n_{s_1}(t^n_{i-1}(s)  )$, i.e $$z_{i} := \Phi( t^n_{i-1}(s) , s/2^n , \gamma^n_{s_1}(t^n_{i-1}(s)    )).$$
Since $\gamma^n_s(t^n_i(s) ) = \Phi( t^n_{i-1}(s) , s/2^n , \gamma^n_s(t^n_{i-1}(s)   ))$ and $z_{i} = \Phi( t^n_{i-1}(s) , s/2^n , \gamma^n_{s_1}(t^n_{i-1}(s)   ))$, by Corollary \ref{g77777},
\begin{equation}\label{e1e2}
\begin{split}
d( z_i , \gamma^n_s(t^n_i(s)  ) )
&\leq e^{ -\lambda_0  s /2^n } d( \gamma^n_{s_1}(t^n_{i-1}(s) ), \gamma^n_s(t^n_{i-1}(s) ) ).
\end{split}
\end{equation}

\textbf{Step 2.} For $d( \gamma^n_{s_1} ( t^n_i(s) ) , z_i)$, let $$y_{j+1} := \Phi \bigl( t_0 + (k +j ) h_1 , ih - (k +j ) h_1 , x_{j + k }(s_1) \bigr) $$ (see Figure \ref{six}). For every $j$, the point $y_{j+1}$ is in time $t_i = t_0 + i h $ since its initial point $x_{j + k}(s_1) =\gamma_1(t_0 + (k +j ) h_1)$ is in time $t_0 + (k +j ) h_1$ and it flows for time $i h  - (k +j ) h_1$.
Note that $y_{ [is/s_1] - k  + 1 } = \gamma^n_{s_1} (t^n_i(s) )$. Also note that if $t_0+ k  h_1 = t_{i-1}$, then $$y_1 = z_i.$$
By Theorem \ref{g9999}, the semigroup property of the fixed-time gradient flows gives that $$y_j =  \Phi \bigl( t_0 + (k + j-1 ) h_1 , ih - (k +j ) h_1 , x_{j + k }(s_1) \bigr).$$
For three points $x_{j+k}(s_1) $, $y_{j}$ and $y_{j+1}$ when $j \geq 1$, from Theorem \ref{g21} at $x_{j+k}(s_1) $ with $T = i h  - (k + j ) h_1 \leq h = s/2^n$ and the time difference $h_1=s_1/2^n$, we have
\begin{equation}\label{e1e4}
 d( y_{j+1} , y_{j} ) \leq 2 B  s/2^n (s_1/2^n)^\alpha .
 \end{equation}
By Theorem \ref{g9999}, the semigroup property of the fixed-time gradient flows gives that $$y_1 =  \Phi \bigl( t_0 + k h_1 , s/2^n , \gamma^n_{s_1}(t^n_{i-1}(s) ) \bigr).$$
For three points $\gamma^n_{s_1}(t^n_{i-1}(s)  )$, $y_{1}$ and $z_i$,
since $z_{i} = \Phi( t^n_{i-1}(s) , s/2^n , \gamma^n_{s_1} (t^n_{i-1}(s)  )) $,
we can apply Theorem \ref{g21} at $\gamma^n_{s_1}( t^n_{i-1}(s) )$ with $T = s/2^n$.
Thus we get
\begin{equation}\label{e1e5}
d( z_{i} , y_{1} ) \leq 2 B s/2^n  (s_1/2^n)^\alpha
 \end{equation}
 since the difference of two fixed-time between the fixed-time gradient from $\gamma^n_{s_1}(t^n_{i-1}(s) )$ to $y_{1}$ and the fixed-time gradient from $\gamma^n_{s_1}(t^n_{i-1}(s) )$ to $z_i$ is less than $s_1/2^n$.

Equations \eqref{e1e4} and \eqref{e1e5} imply that
\begin{equation}\label{e1e3}
\begin{split}
 d( \gamma^n_{s_1} (t^n_i(s)  ) , z_i) &\leq d( \gamma^n_{s_1}(t^n_i(s)  ) , y_{ [is/s_1] - k } ) + \cdots + d( y_{2} , y_1 ) + d( y_1 , z_i)  \\
& \leq 2 B ( [s/s_1] + 1 ) s/2^n (s_1/2^n)^\alpha .
\end{split}
\end{equation}

\textbf{Step 3.}
We are ready for an induction argument to get Equation \eqref{eeeee}. The case $i=1$ is given by \eqref{e1e3} since $z_1 = \gamma^n_s( t^n_1(s)  )$.
Now assume that $$d(\gamma^n_s(t^n_{i-1}(s)  ), \gamma^n_{s_1} (t^n_{i-1}(s)   ) ) \leq 2 B ( [s/s_1] +1) (i-1) s /2^n  e^{ -\lambda_0 s (i-1) / 2^n } ( s_1/ 2^n)^{\alpha} .$$
Then from Equations \eqref{e1e2} and \eqref{e1e3}, we get
\begin{equation*}
\begin{split}
d(\gamma^n_s(t^n_i(s) ), &\gamma^n_{s_1}(t^n_i(s)   ) ) \leq d( \gamma^n_s(t^n_i(s) ),z_{i})+d(z_{i}, \gamma^n_{s_1}(t^n_i(s)   ) ) \\
& \leq d( \gamma^n_s(t^n_i(s)    ), z_{i}) + 2 B ( [s/s_1] +1) s / 2^n ( s_1/ 2^n)^{\alpha}  \\
& \leq e^{ -\lambda_0 s /2^n } d( \gamma^n_s(t^n_{i-1}(s) ), \gamma^n_{s_1}(t^n_{i-1}(s)  ) ) + 2 B ( [s/s_1] +1) s / 2^n ( s_1/ 2^n)^{\alpha}    \\
& \leq e^{ -\lambda_0 s /2^n } 2 B ( [s/s_1] +1) (i-1) s /2^n  e^{ -\lambda_0 s (i-1) / 2^n } ( s_1/ 2^n)^{\alpha}   \\ &+ 2 B ( [s/s_1] +1) s/ 2^n ( s_1/ 2^n)^{\alpha}   \\
& = 2 B( [s/s_1] +1) [ (i-1) s/2^n e^{ -\lambda_0 i s  / 2^n } + s/2^n ] ( s_1/ 2^n)^{\alpha}  \\
& \leq 2 B( [s/s_1] +1) i s /2^n  e^{ -\lambda_0 i s /2^n }  ( s_1/ 2^n)^{\alpha}. \\
\end{split}
\end{equation*}
Letting $i$ be $2^n$, we have
$$d(\gamma^n_s(t_0+s), \gamma^n_{s_1}(t_0+s) ) \leq 2 B( [s/s_1] +1) s  e^{ -\lambda_0  s  }  ( s_1/ 2^n)^{\alpha} .$$
Then taking $n \to \infty$, Claim 3 is proved.


\textbf{Claim 4.} \textit{Limit curves satisfy the semigroup property:} \\\textit{For a limit curve}  $\sigma_{x_0,t_0}$ such that $\sigma_{x_0,t_0}(t_0)= x_0$, \textit{let} $x_1$ \textit{be} $\sigma_{x_0,t_0}(t_0 + s )$ \textit{and} $x_2$ \textit{be} $\sigma_{x_0,t_0}(t_0 + s + s_2)$. \textit{Then} $$x_2 = \sigma_{x_1,t_0 + s}( t_0 + s + s_2 )$$ where $\sigma_{x_1,t_0 + s}$ is the limit curve such that
$\sigma_{x_1,t_0 + s}( t_0 + s  ) = x_1$.

For the proof, we need a piecewise fixed-time gradient curves with step size $h= {s / 2^n}$ from $x_0$ and $t_0$ converging to the limit $\sigma_{x_0,t_0}$ as $n \to \infty$. We can assume that $s_2/s$ is not an integer, i.e $s_2/s > [s_2/s].$
For fixed $s$ and $s_2$, let $$t^n_i(s) = t_0 + i  s  / 2^n $$ for $i=1, \cdots , N(n)$
 and $$t^n_{ N(n) + 1 } = t_0 + s + s_2$$ where $N(n) := 2^n( 1 + [s_2/s] )$. Let $q_0=x_0$ and for $i=1, \cdots , N(n)+1$,
$$q_{i} = \Phi( t^n_{i-1}(s) ,  t^n_i(s) - t^n_{i-1}(s) ,  q_{i-1}  ). $$
Note that Since $t^n_{2_n}(s)=t_0+s$, by Claim 1, $x_1 = \lim_{n \to \infty} q_{2^n}$ and by Claim 3, $x_2 = \lim_{n \to \infty} q_{N(n)+1}$.
Thus we get piecewise fixed-time gradient curves passing through the $q_i$ converging to
the limit curve $\sigma_{x_0,t_0}$ as $n \to \infty$.

For the limit curve $\sigma_{x_1,t_0 + s}$, we need another piecewise fixed-time gradient curves with step size $h= {s / 2^n}$ from $x_1$ and $t_0+s$ converging to
the limit curve $\sigma_{x_1,t_0 + s}$ as $n \to \infty$.
Let $\widehat{p}_{2^n} = x_1$ and $$\widehat{p}_{i} = \Phi( t^n_{i-1}(s),  t^n_i(s) - t^n_{i-1}(s),  \widehat{p}_{i-1}   )$$ for $i = 2^n + 1, \cdots, N(n)+1$.

By Corollary \ref{g77777} (Distance I), we get
$$d(\widehat{p}_{N(n)+1} , q_{N(n)+1} ) \leq  e^{-\lambda s_2 } d(\widehat{p}_{2^n} , q_{2^n} ).$$
Since $\widehat{p}_{2^n}=x_1$ and $x_1 = \lim_{n \to \infty} q_{2^n}$, we get $$\lim_{n \to \infty} d(\widehat{p}_{N(n) + 1} , q_{N(n) + 1} ) =0.$$
Since $x_2 = \lim_{n \to \infty} q_{N(n)+1}$ and $ \sigma_{x_1,t_0 + s}( t_0 + s + s_2 ) =  \lim_{n \to \infty} \widehat{p}_{N(n)+1}$ by Claim 3, the semigroup property of the limit curve is proved.

\textbf{Claim 5.} \textit{Let} $\widetilde{\sigma}$ \textit{be a reparametrization of the limit curve given by} $s \mapsto \widetilde{\sigma}(s):=\sigma_{x_0,t_0}(t_0+t+s)$. \textit{Then for} $s \in I_A$,
 $$d(\widetilde{\sigma}(s),\mu(s)) \leq B' s^{1+\alpha} e^{ -  \lambda_0 s}$$ \textit{where} $\mu=\mu(s)$ \textit{is the fixed-time gradient curve of the function} $x \mapsto F(t_0+t,x)$ \textit{with} $\mu(0)=\widetilde{ \sigma }(0)=\sigma_{x_0,t_0}(t_0+t)$ \textit{and} $\lambda_0 = \min \{ 0,\lambda  \}$.

Note that Claim 5 implies
\begin{equation}\label{eeee1}
\lim_{s \to 0+} {d(\widetilde{\sigma}(s), \widetilde{\sigma}(0) ) \over s} =\lim_{s \to 0+} {d(\mu(s), \mu(0) ) \over s} = || \nabla_{\widetilde{\sigma}(0)} (-F_{t_0+t}) ||.
\end{equation}
The second equality comes from the definition of time-independent gradient curve.

First, we look at $x_0$ and show the case $t=0$. Letting $n=0$ in Claim 2, $p^n_{2^n}$ becomes $\mu(s)$ since $p^n_{2^n}$ is just the time-independent gradient curve of $x \mapsto F(t_0,x)$ when $n=0$. Thus we have $$d( \widetilde{\sigma}(s) , \mu(s) ) \leq  B' s^{1+\alpha} e^{ -\lambda_0 s}  $$  since $\widetilde{\sigma}(s)=\sigma_{x_0,t_0}( t_0+s)$.

Second, when $t>0$, for any point $\widetilde{\sigma}(0) = \sigma_{x_0,t_0}(t_0+t)$, we can do this calculation again since we prove the semigroup property in Claim 4. Then we get the same inequality. Claim 5 is proved.

\textbf{Claim 6.} $\sigma_{x_0,t_0}$ \textit{is a time-dependent gradient curve of} $F$ \textit{at} $x_0$ \textit{and} $t_0$ \textit{such that} $\sigma_{x_0,t_0}(t_0)=x_0$.

For this, we will show the right-side tangent vector of $\sigma_{x_0,t_0}$ at $\sigma_{x_0,t_0}(t_0+t)$ exists and it is equal to the downward gradient vector
$\nabla_{\sigma_{x_0,t_0}(t_0+t)} (- F_{ t_0+ t })$.

Let $\widetilde{\sigma}$ and $\mu$ be as in the Claim 5. Then since $\nabla_{\sigma_{x_0,t_0}(t_0+t)} (- F_{ t_0+ t })$ is the tangent vector of $\mu$ at
 $\mu(0)$,
we need to show that the distance between the direction of $\mu$ at $\widetilde{\sigma}(0)$
and the direction of a geodesic $[\widetilde{\sigma}(0) \widetilde{\sigma}(s)]$ goes to zero as $s \to 0+$.
Indeed, we need to show that the angle $$\lim_{s \to 0+} \angle_{\widetilde{\sigma}(0)} \widetilde{\sigma}(s), \mu(s)$$ is zero. This means that the direction of $\sigma_{x_0,t_0}$ at $\sigma_{x_0,t_0}(t_0+t)$ is equal to the direction of $\mu$ at $\mu(0)$.

When $|| \nabla_{\widetilde{\sigma}(0)} (-F_{t_0+t}) ||$ is zero, it is trivial. So
suppose that $|| \nabla_{\widetilde{\sigma}(0)} (-F_{t_0+t}) ||$ is not zero. So there is a constant $C$ such that $|| \nabla_{\widetilde{\sigma}(0)} (-F_{t_0+t}) || > 2C$. For sufficiently small $s$, we get $d( \mu(s) , \mu(0) ) \geq s C $ and $d( \widetilde{\sigma}(s) , \widetilde{\sigma}(0) ) \geq s C $, where $\mu(0)=\widetilde{\sigma}(0)$,
since $$\lim_{s \to 0+} {d ( \mu(s), \mu(0)  ) \over s} = || \nabla_{\widetilde{\sigma}(0)} (-F_{t_0+t}) || = \lim_{s \to 0+} {d ( \widetilde{\sigma}(s), \widetilde{\sigma}(0)  ) \over s}.$$
By Claim 5, this implies that $\angle_{\widetilde{\sigma}(0)} \widetilde{\sigma}( s ), \mu( s )$ is less than the angle of a Euclidean triangle with two edge lengths $sC$, $sC$ and third edge length less than  $  B' s^{1+\alpha}  e^{ -  \lambda_0 s } $. Then it implies that $\angle_{\widetilde{\sigma}(0)} \widetilde{\sigma}( s ), \mu( s )$ is less than $ {B's^\alpha \over C}e^{ -  \lambda_0 s } $. As $s \to 0+$, this becomes zero. Therefore $$\lim_{s \to 0+} \angle_{\widetilde{\sigma}(0)} \widetilde{\sigma}(s), \mu(s) = 0.$$
So two directions are same.
Since by Equation \eqref{eeee1}, the speed of the right-side tangent vector is same as the length of gradient vector, Claim 6 is proved.

\textbf{Claim 7.} $\sigma_{x_0,t_0}$ \textit{is the unique time-dependent gradient curve of} $F$ \textit{at} $x_0$ \textit{and} $t_0$.

For this, we will show that for any time-dependent gradient curves $\sigma_i$ of $F$ at $x_i$ and $t_0$ such that $\sigma_i(t_0)=x_i$ where $i=1,2$,
$$d( \sigma_1(t_0+s) , \sigma_2(t_0+s) ) \leq e^{ - \lambda s} d(x_1,x_2).$$
Then clearly, this estimate implies that any two time-dependent gradient curves with same initial position and time are same. This yields uniqueness of $\sigma_{x_0,t_0}$.

Let $f(s)$ be $d( \sigma_1(t_0+ s), \sigma_2(t_0+ s) )$.
By Lemma \ref{l222} with $x=\sigma_1(t_0 + s)$ and $y=\sigma_2( t_0+s)$, we have
$$ \langle \xi_1(s)  ,  \nabla_{ \sigma_1(t_0 + s) } (- F_{ t_0+ s }) \rangle + \langle  \xi_2(s) , \nabla_{ \sigma_2(t_0 + s) } (- F_{ t_0+ s }) \rangle   \geq  \lambda f(s)$$
where $\xi_1(s)$ is the direction of $[\sigma_1( t_0 + s) \sigma_2( t_0+ s )]$ and $\xi_2(s)$ is the direction of $[\sigma_2( t_0+s) \sigma_1( t_0+s)]$.
By first variation formula, we get $$ f^+(s) \leq - \lambda f(s).$$
Solving this, Claim 7 is proved.
\end{proof}

From the proof of Theorem \ref{g7}, we extract the following corollaries.
\begin{cor}\label{g8}
For any time-dependent gradient curves $\sigma_i$ of $F$ at $x_i$ and $t_0$ such that $\sigma_i(t_0)=x_i$ where $i=1,2$,
$$d( \sigma_1(t_0+s) , \sigma_2(t_0+s) ) \leq e^{ - \lambda s} d(x_1,x_2).$$
\end{cor}
\begin{proof}
This is Claim 7 in the proof of Theorem \ref{g7}.
\end{proof}

\begin{cor}\label{g10}
For any time-dependent gradient curve $\sigma_{x_0,t_0}$ and its piecewise fixed-time gradient curves of $p^n_{i}$ from Theorem \ref{g7},
$$d( \sigma_{x_0,t_0}( t_0+s ) , p^n_{2^n} ) \leq B' s^{1+\alpha} \tfrac{1}{2^{\alpha n}} e^{ -\lambda_0 s} $$ where $\lambda_0 = \min \{ 0,\lambda  \}$.
\end{cor}
\begin{proof}
This is Claim 2 in the proof of Theorem \ref{g7}.
\end{proof}

\begin{cor}\label{g888888}
For any time-dependent gradient curve $\sigma_{x_0,t_0}$, let $x_1$ be $\sigma_{x_0,t_0}(t_0 + s)$ and $x_2$ be $\sigma_{x_0,t_0}(t_0 + s + s_2)$. Then $$x_2 = \sigma_{x_1,t_0 + s}( t_0 + s + s_2 ).$$
\end{cor}
\begin{proof}
This is Claim 4 in the proof of Theorem \ref{g7}.
\end{proof}

\begin{examp}\label{ex11111111}
Let $X = \mathbb{R}$ and $F: \mathbb{R}_{\geq 0} \times X \to \mathbb{R}$ given by $$F(t,x)=-(t+1)x.$$
Then a downward gradient vector $\nabla_x ( - F_t)$ of $F_t$ at $x$ is $(t+1)$.
This example satisfies the conditions of Theorem \ref{g7}. In particular, $$\rho( \nabla_x ( -F_t) , \nabla_x ( -F_{t'}) ) = |t-t'|.$$
So we can get a time-dependent gradient curves $\sigma_{x_0}:[0, \infty) \to X$ given by $$\sigma_{x_0}(t) = 1/2 t^2 + t + x_0.$$
Let us look at the gradient flow $\Phi: \mathbb{R}_{\geq 0} \times X \times [0,1] \to X$ of $F$. It is given by $$\Phi(t,x_0,h) = (t+1) h + x_0.$$ Then a value $\rho(x_0,t):= \sup_{h \neq k} { d( \Phi(t,x_0,h) , \Phi(t,x_0,k) )   \over |h-k|     }$ is equal to $t+1$.
But since $\rho(x_0,t)$ is not bounded in $t$, $\Phi$ does not have linear speed growth (see the following definition for linear speed growth ).
So we have a time-dependent gradient curve $\sigma_{x_0}$ defined on $[0,\infty)$ even if $\Phi$ doesn't have linear speed growth.
\end{examp}

\begin{defn}
Given a flow $\Phi: \mathbb{R}_{\geq 0} \times X \times [0,1] \to X$,
and fixed $t,r,l$ and $x$, let $$\rho(x,t;r,l):= \sup_{ d(x,y)<r , |t-s| \leq l} \rho(y,s)$$ where $\rho(x,t):= \sup_{h \neq k} { d( \Phi(t,x,h) , \Phi(t,x_0,k) )   \over |h-k|     }$.

Then $\Phi$ is said to have \emph{linear speed growth} if there is a point $x \in X$ with positive constants $c_1(x)$ and $c_2(x)$ such that for all $r>0$ and $t,l >0$,
$$\rho(x,t;r,l) \leq c_1(x) r + c_2(x).$$
\end{defn}

\section{Application: Simple pursuit}
\subsection{Simple pursuit curves chasing convex sets}
Suppose $(X,d)$ is a \catzero\ space.
We use the Hausdorff metric $d_H$ on the set of closed convex sets.

Let $d_Y: X \to \mathbb{R}$ be the distance to a complete convex subset $Y \subset X$, defined by $$x \mapsto d_Y(x):= \min_{y \in Y} d(y,x).$$
\begin{prop}(See \cite[Page 178]{bh})\label{de2}
Let $Y$ be a complete convex subset. Then
the distance function
$d_Y$ is convex and for $x \in X$, there is a unique point $x_0$ of $Y$ such that $d_Y(x)=d(x_0,x)$.
\end{prop}
This point $x_0$ is called the \emph{footpoint of $x$ in $Y$}.




\begin{prop}\label{ppp2}
Let $Y_t$ be a curve of closed convex sets in $X$ such that $d_H(Y_t,Y_{t'}) \leq |t - t'|$ and $q_t$ be the footpoint of $p$ in $Y_t$.
Then we have $$|d(p,q_t) - d(p,q_{t'})| \leq |t-t'|$$ and $$ d(q_t,q_{t'})  \leq \bigl(\sqrt{|t - t'|} + 2 \sqrt{ d(p,q_t)+ |t-t'| } \;\bigr)\sqrt{|t-t'|}$$ for any $t$ and $t'$.
\end{prop}

\begin{proof}
Let $x$ be the footpoint of $q_t$ in $Y_{t'}$, and $y$ be the footpoint of $q_{t'}$ in $Y_t$.
For all $\epsilon > 0$, $$d( x, q_t) \leq |t-t'| + \epsilon$$ and $$d(y ,q_{t'}) \leq |t-t'| + \epsilon$$ since $d_H(Y_t,Y_{t'}) \leq |t - t'|$. Hence,
\begin{equation}\label{q11}
d(x,q_t) \le |t-t'| \;\; and \;\; d(y,q_{t'}) \le |t-t'|.
\end{equation}
By the triangle inequality, we have
\begin{equation}\label{q12}
d(p ,x ) - d(p ,q_t) \leq |t-t'|
\end{equation}
and
\begin{equation}\label{q13}
d(p , y) - d(p ,q_{t'}) \leq |t-t'|.
\end{equation}
By the definition of $q_{t'}$, we get
$$ d( p, q_{t'}) \leq d(p,x) \leq d(p ,q_t) + |t-t'|.$$
Similarly, by the definition of $q_{t}$, $$ d(p ,q_{t}) \leq d(p,y) \leq d(p, q_{t'}) + |t-t'|.$$
This means
\begin{equation}\label{q14}
 \bigl|  d(p ,q_t) - d(p ,q_{t'}) \bigr| \leq |t-t'|.
\end{equation}

Let $z_s$ be the point on $[z_0 z_1]$ such that $sd( z_0,z_1) = d(z_0 ,z_s)$ and $(1-s)d(z_0,z_1) = d(z_1,z_s)$ for $0 \leq s \leq 1$.
For $[q_t y]$, suppose that  $z_0 = q_t$, and $z_1 = y$.
Since $q_t$ and $y$ are in $Y_t$, $[q_t y]$ lies in $Y_t$.
Then, by Proposition \ref{g1}, we get
$$d( p, q_t)^2 \leq d(p ,z_s)^2 \leq (1-s) d(p ,q_t)^2 + s d(p, y)^2 -s(1-s) d( q_t ,y)^2.$$
This gives $$s(1-s) d( q_t, y)^2 \leq   s d(p, y)^2 - s d( p, q_t)^2.$$
Dividing both sides by $s$ and taking $s \to 0$, we have $$d( q_t ,y)^2 \leq    d(p, y)^2 - d(p ,q_t)^2.$$
Then it becomes $$ d (q_t ,y)^2 \leq (d(p,y)+d(p,q_t))(d(p,y) - d(p,q_t)).$$
With \eqref{q13} and \eqref{q14}, it yields $$ d(q_t ,y)^2 \leq (2 d(p,q_t)+ 2|t-t'|  )( 2 |t-t'| ).$$
Then $$ d(q_t ,y) \leq 2 \sqrt{ d(p,q_t)+ |t-t'| } \sqrt{ |t-t'| }.$$
Therefore by \eqref{q11}, we have $$ d( q_t ,q_{t'}) \leq |t - t'|  + 2 \sqrt{ d( p,q_t) + |t-t'| } \sqrt{ |t-t'| }. $$
\end{proof}

\begin{figure}[h]
\centering{\epsfig{file=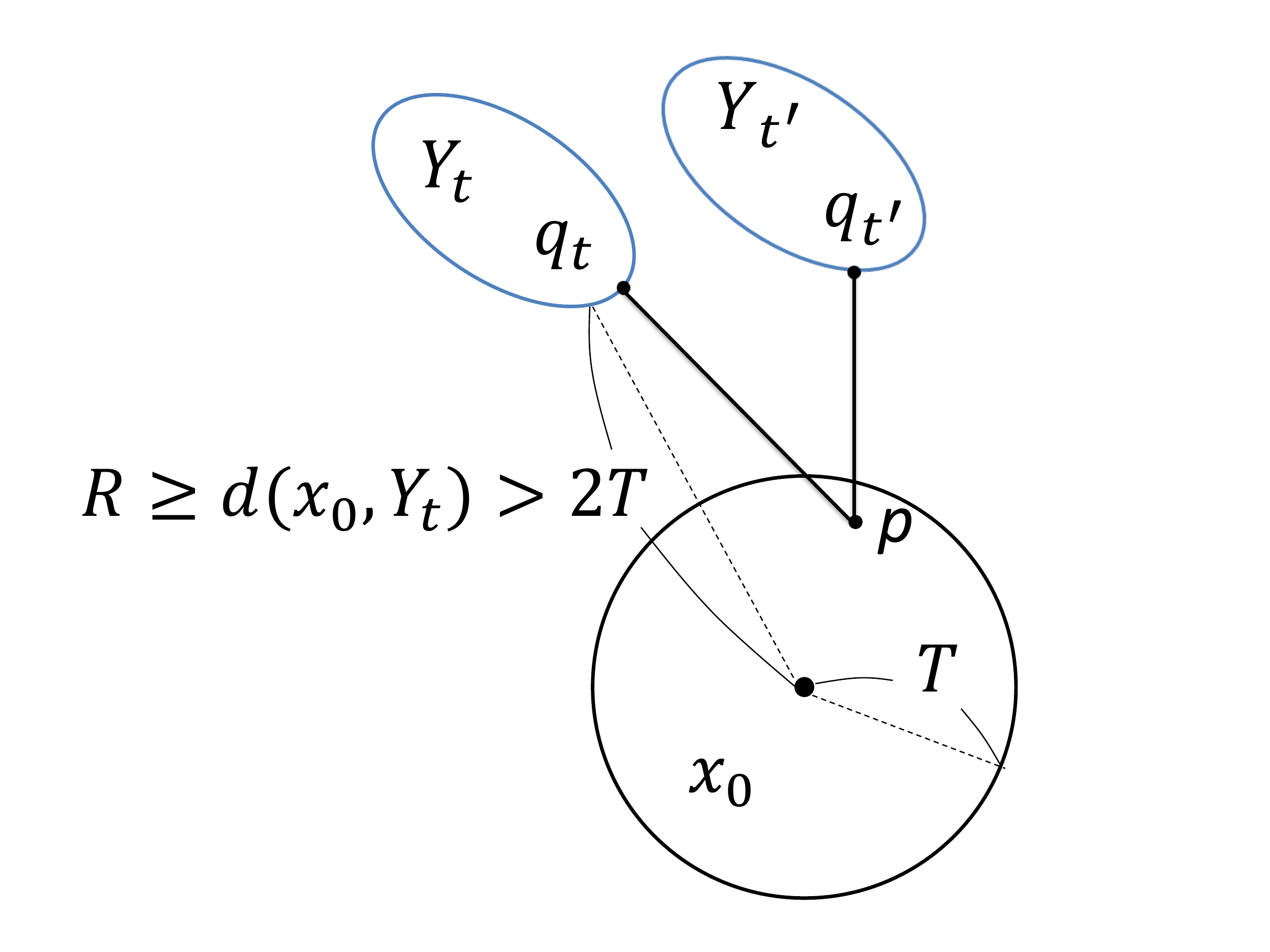  , height=6.4cm, width=9cm}}
\caption{ The footpoint $q_t$ of $p$ in $Y_t$ and the closed ball $B_{x_0}(L)$ }
\end{figure}

\begin{theorem}[ $\alpha= 1/2$ ]\label{ppp3}
Let $Y_t$ be a curve of closed convex sets in $X$ such that $t \in [0, \infty)$ and $d_H(Y_t,Y_{t'}) \leq |t - t'|$.
For any $x_0 \in X$,
if there is a constant $T>0$ such that $\min_{t \in [0,T]} d_{Y_t}(x_0) > 2T$,
then inside the closed ball $B_{x_0}(T)$, there is a unique time-dependent gradient curve $\sigma_{x_0}:[0,T] \to X$
of the function $F$ where $F(t,p):=d_{Y_t}(p)$ such that $\sigma_{x_0}(0)= x_0$.
\end{theorem}
$\sigma_{x_0}$ is called a \emph{simple pursuit curve} chasing the convex set $Y_t$ with an initial $x_0$.
\begin{proof}
If it exists, its speed is one. So it must be inside the ball $B_{x_0}(T)$.
For $x,z \in X$,
let $x_t$ be the footpoint of $x$ in $Y_t$ and $z_t$ be the footpoint of $z$ in $Y_t$.
Suppose that $d(x ,x_t) \leq d(z , z_t)$. Let $z'$ be the point on $[z z_t]$ at distance $d(x , x_t)$ from $z_t$.
Since $z'$ is the footpoint of $z$ in $\{ y \in X \; | \; d_{Y_t}(y) \leq d( x, x_t)  \}$,  $$d(z' ,z) \leq d( x ,z ).$$ Then we have $$F(t,z) - F(t,x) = d( z, z_t) - d( x , x_t) = d(z' ,z) \leq d( x ,z ).$$
Thus $F$ is 1-Lipschitz on $X$.

From Proposition \ref{de2}, $F$ is $\lambda$-convex on $X$ with $\lambda =0$. By Proposition \ref{ppp2}, we know that $F$ is L-Lipschitz in $t$ with $L=1$. Thus we see that $F$ satisfies three conditions of Theorem \ref{g7}.

Next, we will show that $F$ satisfies last condition of Theorem \ref{g7}.
Let $R = R(T,x_0) := \max_{t \in [0,T]} d_{Y_t}(x_0)$. Then $R > 2T$.

For $t, t' \in [0,T]$,
suppose that $|t - t'| \leq R$ and $5\sqrt{R} \sqrt{|t - t'|}$ is smaller than $T/2$.
For a point $p \in B_{x_0}(T)$,
 let $q_t$ be the footpoint of $p$ in $Y_t$ and $q_{t'}$ be the footpoint of $p$ in $Y_{t'}$.

Then by triangle inequality,
$$ d(p,q_t) + d(p,x_0) \geq d(x_0, q_t) \geq d(x_0,Y_t) > 2T.$$
Since $d(p,x_0) \leq T$, this implies $d(p,q_t) > T$. Similarly, we get $d(p,q_{t'}) > T$.

Also $d( p, q_t) \leq d(p,x_0) + d(x_0, Y_t)\leq T+R \leq 2R$.
Then from Proposition \ref{ppp2},
\begin{equation}\label{eee1}
\begin{split}
d( q_t,  q_{t'}) &\leq \bigl(\sqrt{|t - t'|} + 2 \sqrt{ d(p,q_t)+ |t-t'| } \;\bigr)\sqrt{|t-t'|}  \\
&\leq  (\sqrt{R} + 2 \sqrt{ 2R + R  } ) \sqrt{|t-t'| }  \\
&<   5\sqrt{R} \sqrt{|t - t'|}.
\end{split}
\end{equation}
Thus $d( q_t, q_{t'}) \leq T/2$.


If we consider a comparison triangle $\triangle \widetilde{q}_t \widetilde{p}  \widetilde{q}_{t'}$, since $d( \widetilde{p}, \widetilde{q}_t) > T$, $$\sin \angle \widetilde{q}_t \widetilde{p} \widetilde{q}_{t'} \leq d( \widetilde{q}_t, \widetilde{q}_{t'}) / T.$$
Since $d(\widetilde{q}_t, \widetilde{q}_{t'}) / T = d( q_t, q_{t'}) / T \leq 1/2$,
this implies
\begin{equation}\label{eee2}
\angle \widetilde{q}_t \widetilde{p} \widetilde{q}_{t'} \leq {2 d( \widetilde{q}_t, \widetilde{q}_{t'} ) \over  \sqrt{3}T}
\end{equation} because $\sin^{-1}$ has Lipschitz constant $2/\sqrt{3}$ on $[0,1/2]$.

By Equations \eqref{eee1} and \eqref{eee2}, we have  $$\angle \widetilde{q}_t \widetilde{p} \widetilde{q}_{t'}  \leq { 10 \sqrt{R} \over  \sqrt{3} C } \sqrt{|t-t'|  }.$$

Since $|| \nabla_{p} (-F_t) || = 1$, it implies $$ \rho( \nabla_{p} (-F_t) , \nabla_{p} (-F_{t'}) )  = \angle q_t p q_{t'} \leq  { 10 \sqrt{R} \over  \sqrt{3} C } \sqrt{|t-t'|  }. $$

By Theorem \ref{g7} with $\alpha=1/2$, $B = { 10 \sqrt{R} \over  \sqrt{3} C }$, $B_0 = \min \{ R , T^2 / 100 R \}$ and $X=B_{x_0}(T)$, we have a time-dependent gradient curve starting at $x_0$.
\end{proof}

\begin{theorem}
Let $Y_t$ be a curve of closed convex sets in $X$ such that $t \in [0, \infty)$ and $d_H(Y_t,Y_{t'}) \leq |t - t'|$. Then for any $x_0 \in X$,
 there is a unique time-dependent gradient curve $\sigma_{x_0}$
of the function $F$ where $F(t,p):=d_{Y_t}(p)$ such that $\sigma_{x_0}(0)= x_0$.
When it meets the curve of closed convex sets in $X$ at time $t_1$, that is, $d(Y_{t_1},\sigma_{x_0}(t_1))=0$, it will stop.
\end{theorem}
\begin{proof}

For the initial point $x_0$, let $f(s) = \min_{t \in [0,s ]} d_{Y_t}(x_0)$.
If there is a constant $T_{x_0}>0$ such that $\min_{t \in [0,T_{x_0} ]} d_{Y_t}(x_0) > 2 T_{x_0}$, by Theorem \ref{ppp3},
then inside the closed ball $B_{x_0}(T_{x_0})$, there is a unique time-dependent gradient curve $\sigma_{x_0}:[0,T_{x_0}] \to X$
of the function $F$ where $F(t,p):=d_{Y_t}(p)$ such that $\sigma_{x_0}(0)= x_0$.
If there is no such constant $T_{x_0}$, it means that $d(x_0,Y_0) = 0$ since $f(s_1) \geq f(s_2)$ when $0 \leq s_1 \leq s_2$ and $f$ is continuous.
When it arrives at $x_1=\sigma_{x_0}(T_{x_0})$, we find a constant $T_{x_1}$ again.
If there is a constant $T_{x_1}$ for $x_1$, by Theorem \ref{ppp3}, we can extend our gradient curve
at $x_1$ or it stops at $x_1$.
This means that we will have an interval $[0,t_1]$ such that the time-dependent gradient curve $\sigma_{x_0}$ exists on $[0,t_1]$ and $d( \sigma_{x_0}(t_1) , Y_{t_1})=0$ or $\sigma_{x_0}$ exists on infinite time.
\end{proof}

For $Y_t$, if $Y_t$ is just a point of $X$, then we have the following corollary.
\begin{cor}\label{c1}
Let $Y_t$ be a curve of points with speed $\leq 1$. Then for any $x_0 \in X$, there is a unique simple pursuit curve $\sigma_{x_0}$ chasing the curve $Y=Y(t)=Y_t$ such that $\sigma_{x_0}(0)=x_0$.
\end{cor}

In \cite{j2}, we study continuous pursuit curves chasing a moving point on \catk\ spaces and geometrically show existence and uniqueness of continuous pursuit curves on \catk\ spaces. Additionally, we get regularity of continuous pursuit curves which is a replacement for $C^{1,1}$ regularity in smooth spaces.

\subsection{Simple pursuit curve chasing barycenters}
Suppose that $X$ is a \catzero\ space.
When we deal with multiple evaders in pursuit-evasion games, we still may find a strategy to get a continuous pursuit curve.
For this, we consider the \emph{barycenter} of multiple points.

Since for points $x_1, \cdots, x_n \in X$, the function $x \mapsto \sum_{i=1}^n d^2( x , x_i)$ is strictly convex, there exists a unique minimum point of this function.   
\begin{defn}\cite[Def. 2.3]{j} Given $n$ points $x_i$ on $X$, the \emph{barycenter} $b$ of the $x_i$'s is defined to be the minimum point of the function $x \mapsto \sum_{i=1}^n d^2( x , x_i)$.
\end{defn}

Let $\mathcal{P}(X)$ denote the set of all probability measures $\nu$ on $(X,\mathcal{B}(X))$ with separable support supp($\nu$) $\subset X$ where $\mathcal{B}(X)$ is the set of Borel sets of $X$.

More generally,
\begin{defn}\cite[Prop. 4.3]{s2}
For $\nu \in \mathcal{P}(X)$ such that $\int_X d^2(x,y) \nu(dy) < \infty$ for some (hence all) $x \in X$, the minimum point of the function $x \mapsto \int_X d^2(x,y) \nu(dy)$ is called the \emph{barycenter} $b(\nu)$ of $\nu$.
\end{defn}
Since the function $x \mapsto \int_X d^2(x,y) \nu(dy)$ is continuous and strictly convex, the barycenter of $\nu$ is well-defined.

Let $\delta_x$ be the \emph{Dirac measure} of $x$ given by $\delta_x(A)=1$ if $x \in A$ or $\delta_x(A)=0$ otherwise, for any subset $A$ of $X$.
\begin{lemma}\label{bb1}
Set $\nu = {1 \over n} \sum_{i=1}^n \delta_{x_i}$. Then the barycenter $b(\nu)$ of $\nu$ is equal to the barycenter of the $x_i$'s.
\end{lemma}
\begin{proof}
Since $\nu = {1 \over n} \sum_{i=1}^n \delta_{x_i}$,
$$\int_X d^2(x,y) \nu(dy) = {1 \over n}\sum_{i=1}^n d^2( x , x_i).$$ Then the minimum point of the function $x \mapsto \sum_{i=1}^n d^2( x , x_i)$ is equal to the barycenter of $\nu$.
\end{proof}

Here we want to give a strategy for chasing multiple evaders. Given $n$ evaders $E_i=E_i(t)$ with speed $\leq 1$ in $X$, we have the barycenter curve $b=b(t)$ defined by $b(t):=$ the barycenter of $E_i(t)$. Then we want to show that $b=b(t)$ has also speed $\leq 1$.

In order to show this, we need a theorem to deal with the distance between $b(t)$ and $b(t')$. From \cite{s2}, we have the following theorem.
This theorem gives an upper bound of the distance between two barycenters by integrating a coupling.
Given two probability measures $\nu_1, \nu_2 \in \mathcal{P}(X)$, we call $\widetilde{\nu} \in \mathcal{P}(X^2)$ a \emph{coupling} of $\nu_1$ and $\nu_2$ if $\widetilde{\nu}(A \times X) = \nu_1(A)$ and $\widetilde{\nu}(X \times A) = \nu_2(A)$ for $\forall A \in \mathcal{B}(X)$.

\begin{theorem}\cite[Th. 6.3]{s2}\label{bb2}
Let $X$ be a \catzero\ space.
If two probability measures $\nu_1$ and $\nu_2$ satisfy $\int_X d(x_0,y) \nu_i(dy) < \infty$ for some $x_0 \in X$ and  $\widetilde{\nu}$ is a coupling of $\nu_1$ and $\nu_2$, then
$$ d( b( \nu_1), b( \nu_2) ) \leq \int_{X^2} d(x,y) \widetilde{\nu}(dx dy).$$
\end{theorem}

\begin{prop}\label{bb3}
Let $X$ be a \catzero\ space. Given $n$ evaders $E_i=E_i(t)$ with speed $\leq 1$ in $X$,
then the barycenter curve $b=b(t)$ is 1-Lipschitz.
\end{prop}
\begin{proof}
Let $\widetilde{\nu}$ be ${1 \over n} \sum^n_{i=1} \delta_{(E_i(t) , E_i(t') )}$. Then $\widetilde{\nu}$ is the coupling of $\nu_t$ and $\nu_{t'}$ where $\nu_t = {1 \over n} \sum_{i=1}^n \delta_{E_i(t)}$ and $\nu_{t'} = {1 \over n} \sum_{i=1}^n \delta_{E_i(t')}$.

Obviously, $\nu_i$ satisfies $\int_X d(x_0,y) \nu_i(dy) < \infty$ for some $x_0 \in X$.

Since $b(t)=b(\nu_t)$ and $b(t')=b(\nu_{t'})$ from Lemma \ref{bb1}, by applying $\nu_t$, $\nu_{t'}$ and $\widetilde{\nu}$ to Theorem \ref{bb2}, we have
$$ d(b(t),b(t')) \leq {1 \over n} \sum^n_{i=1} d(E_i(t) , E_i(t') ).$$
Since each evader $E_i$ has speed $\leq 1$, $d(E_i(t) , E_i(t') ) \leq |t - t'|$. Then we have $$d(b(t),b(t')) \leq |t-t'|.$$
\end{proof}
This proof shows that the barycenter curve of curves with speed $\leq 1$ has also speed $\leq 1$.
By letting $Y$ of Corollary \ref{c1} be a barycenter curve $b$, we obtain
\begin{theorem}
Let $X$ be a \catzero\ space. Given $n$ evaders $E_i=E_i(t)$ with speed $\leq 1$ in $X$, there is a unique continuous pursuit curve $P=P(t)$ chasing the barycenter curve $b$ of evaders.
\end{theorem}
\begin{proof}
Since $b$ has speed $\leq 1$ by Proposition \ref{bb3}, we take the barycenter curve $b$ as $Y$ of Corollary \ref{c1}.
\end{proof}

\section{Lytchak's gradient curves}
In this section we consider time independent $F$ and two definitions of gradient curves.  One is Lytchak's definition (see \cite{l}) and the other is the definition we have been using.  We first introduce Lytchak's definition which applies in a very general setting.  We then see that Mayer gets Lytchak's gradient curves. We then show that in our setting Lytchak's gradient curves are in fact gradient curves in the sense we have used them.

Let us start by defining an \emph{absolute gradient} of $F$ at $x$ from \cite{m} and \cite{p2} for the downward case.
\begin{defn}\label{aa1}
Let $X$ be a \catzero\ space.
For a function $F:X \to \mathbb{R}$ and $x \in X$, define the \emph{absolute gradient} $|\nabla_{-}F|(x)$ of $F$ at $x$ by $$|\nabla_{-}F|(x) := \max \Bigl\{ \limsup_{ y \to x} \frac{ F(x) - F(y)}{d(x,y)} ,0 \Bigr\}. $$
\end{defn}
The following condition is sufficient for the set $\{ x \in X : |\nabla_{-}F|(x) \ne 0 \}$ of non-critical points to be open:
\begin{defn}\cite{l},\cite{p2}
Let $X$ be a \catzero\ space.
For a function $F:X \to \mathbb{R}$, $F$ has \emph{semi-continuous absolute gradients} if $\liminf_{ y \to x}|\nabla_{-}F|(y) \geq |\nabla_{-}F|(x)$ for all $x \in X$.
\end{defn}

By Definition \ref{aa1}, we know:
\begin{lemma}
If $F$ is locally Lipschitz and $\lambda$-convex on a \catzero\ space $X$, then $|| \nabla_x (-F) || = |\nabla_{-}F|(x)$ for any $x \in X$.
\end{lemma}

Now, we can give the definition of \emph{gradient curves} on metric spaces.
\begin{defn}\label{g4}\cite{l}
Let $X$ be a \catzero\ space.
For a function $F:X \to \mathbb{R}$ having semi-continuous absolute gradients, a curve $m:[0,a) \to X$ is called the \emph{(time-independent) gradient curve} of $F$ if for all $t \in [0,a)$,
\begin{equation}\label{e21}
\lim_{\epsilon \to 0+} \frac{ d(m(t+ \epsilon ),m(t) )}{ \epsilon} = |\nabla_{-}F|(m(t))
\end{equation}
and
\begin{equation}\label{e22}
  \lim_{\epsilon \to 0+} \frac{ F\circ m(t+\epsilon )-F \circ m(t) }{ \epsilon}= - \bigl(|\nabla_{-}F|(m(t))\bigr)^2.
\end{equation}
\end{defn}

Next we look at Mayer's Theorem from \cite{m}:
\begin{theorem}\label{g66}\cite[Th. 1.13, Th. 2.17, Cor 2.18 and Prop. 2.25]{m}
Let $X$ be a \catzero\ space.
For $x_0 \in X$ and a function $G:X \to \mathbb{R}$, assume that \\
1) $G$ is lower semicontinuous, \\
2) $G$ is $\lambda$-convex. \\
For any $y \in X$, let
\begin{equation}\label{e2222222222222}
\begin{split}
A &= - \min \{0, \liminf_{ d(x,y) \to \infty}  \frac{G(x)}{d^2(x,y)}  \}, \\
I_A &=  \left\{ \begin{array}{ll}
 (0,\infty)& \mbox{ if $A=0$ },  \\
 (0,\frac{1}{16A}]& \mbox{ if $A>0$ }.
\end{array}
\right.
\end{split}
\end{equation}
Then there is a unique (time-dependent) gradient curve $m:I_A \to X$ of $G$ as in Definition \ref{g4} such that $\lim_{t \to 0} m(t) = x_0$ and $G(m(t)) \leq G(x_0)$.
\end{theorem}
Note that $A$ is independent of $y$ because of the triangle inequality.

In Theorem \ref{g66}, it may not happen that the curve $m$ has right-side tangent vectors $m'(t+)$ for all $t \in I_A$. But if $G$ is not only lower semicontinuous but in fact locally Lipschitz, we show below that the curve $m:I_A \to X$ which we get by Theorem \ref{g66} is the gradient curve such that there exists its right-side tangent vector at every time and it must be equal to a gradient vector.
\begin{prop}
Let $X$ be a \catzero\ space and $G:X \to \mathbb{R}$ be locally Lipschitz and $\lambda$-convex.
Then for the curve $m:I_A \to X$ which we get by Theorem \ref{g66},
 there exists a right-side tangent vector $m'(t+)$ at $t$ and it is equal to $\nabla_{m(t)}(-G)$ for all $t \in I_A$.
\end{prop}
\begin{proof}
For $x:= m(t)$, we have the unique downward gradient $\nabla_x (-G) \in T_x$ from Lemma \ref{g11}.
Then let $v$ be the gradient $\nabla_{x}(-G)$ and $w_i$ be the tangent $\in T_{x}$ of a geodesic $[m(t) m(t_i)]$ for any $t_i > t$.
Then it follows from \eqref{e21} and \eqref{e22} that as $i \to \infty$,
\begin{equation}\label{zz1}
||w_i|| \to |\nabla_{-}G|(x) \;\;\text{and}\;\; d_{x} G(w_i) \to - (|\nabla_{-} G|(x))^2.
\end{equation}
If $v= o_x$, our proof is done since $||w_i||$ goes to $0$.
Otherwise, by Definition \ref{gg2}, we get
$$d_{x} G(w_i) \geq - \langle w_i, v \rangle = - ||w_i|| || v || \cos \theta_i \;\; \text{and} \;\; d_x G( v ) = - || v ||^2$$ where $\theta_i$ be the angle between $w_i$ and $v$.
Then
we obtain $$ -\frac{ d_{x} G(w_i) }{  ||w_i|| || v || } \leq \cos \theta_i. $$
Since $|| v || = |\nabla_{-}G|(x)$ by Definition \ref{aa1}, the left side becomes $1$ by \eqref{zz1}. Then $\theta_i$ goes to zero as $i \to \infty$. Our proof is finished.
\end{proof}

\section*{Acknowledgements}
\small The author would like to express his gratitude to Prof. S. Alexander and Prof. R. Bishop for their suggestions and
advice. He appreciates valuable comments from Prof. C. Croke, Prof. V. Kapovich, Prof. A. Lytchak and D. Lipsky. He is grateful to Prof. R. Ghrist for his support during the preparation of this article.
The material in this paper is part of the author's thesis \cite{j1}.
He gratefully acknowledge support from the ONR Antidote MURI project, grant no. N00014-09-1-1031.

\appendix\label{appendix}

\bigskip
\bigskip
\end{document}